\documentclass[preprint,12pt]{elsarticle}
\usepackage{amssymb}
\usepackage{amsmath}
\usepackage[numbers]{natbib}
\usepackage{xcolor}
\usepackage[colorlinks,citecolor=blue]{hyperref}
\usepackage{enumitem}
\usepackage[caption=false]{subfig} 
 \newtheorem{thm}{Theorem}
 \newtheorem{lem}[thm]{Lemma}
 \newtheorem{prop}[thm]{Proposition}
 \newtheorem{cor}[thm]{Corollary}
 \newdefinition{defn}{Definition}
 \newdefinition{rem}{Remark}
 \newproof{proof}{Proof}
 \makeatletter
\def\ps@pprintTitle{%
  \let\@oddhead\@empty
  \let\@evenhead\@empty
  \def\@oddfoot{\reset@font\hfil\thepage\hfil}
  \let\@evenfoot\@oddfoot
}
\makeatother
\begin{document}
\begin{frontmatter}
\title{Iterations of Meromorphic Functions involving Sine}
\author[1]{Gaurav Kumar \corref{cor1}}
\ead{gaurav\_kumar@iitg.ac.in}
\author[1]{M. Guru Prem Prasad}
\ead{mgpp@iitg.ac.in}
\cortext[cor1]{Corresponding author}
\address[1]{{Department of Mathematics, Indian Institute of Technology Guwahati}, 
  {Assam},
  {India}}
\begin{abstract} 
 In this article, the dynamics of a  one-parameter family of functions $f_{\lambda}(z) = \frac{\sin{z}}{z^2 + \lambda},$  $\lambda>0$, are studied. It shows the existence of parameters $0< \lambda_{1}< \lambda_{2}$ such that bifurcations occur at $\lambda_1$ and $\lambda_2$ for $f_{\lambda}$. It is proved that the Fatou set $\mathcal{F}(f_{\lambda})$ is the union of basins of attraction in the complex plane for $\lambda \in (\lambda_1, \lambda_2) \cup (\lambda_2, \infty)$.
Further, every Fatou component of $f_{\lambda}$ is simply connected for $\lambda \geq \lambda_1$. The boundary of the Fatou set $\mathcal{F}(f_{\lambda})$ is the Julia set  $\mathcal{J}(f_{\lambda})$ in the extended complex plane for $\lambda> 1$. Interestingly, it is found that $f_{\lambda}$ has only one completely invariant Fatou component, say $U_\lambda$ such that  $\mathcal{F}(f_{\lambda}) = U_{\lambda}$ for $\lambda >\lambda_2$. Moreover, the characterization of the Julia set of $f_{\lambda}$ is seen for $\lambda \in (\lambda_1, \infty)\setminus \{\lambda_2\}$.\\
\textbf{Keywords}: Transcendental meromorphic function, basin of attraction, completely invariant component.\\
\textbf{Mathematics Subject Classification:} Primary: 37F10, 30D05.
\end{abstract}
\end{frontmatter}
\section{Introduction}
Once Kepler's laws of Earth's motion around the Sun were successfully explained, all efforts to find an analytical solution to the three-body problem ended in failure. In 1890, H. Poincare studied the solar systems of three bodies and concluded that the motions were very complicated. Instead of finding explicit analytic solutions of differential equations corresponding to the three-body problem, Poincaré worked on the qualitative behavior of these solutions. Further, his efforts to understand the qualitative behavior of systems gave rise to the development of the “chaos theory”.
\par Rather than focusing on the continuous trajectories of a differential equation, the theory of dynamical systems often prefers to consider their discrete analogues $\{f^n(w)\}_{n=0}^{\infty}$, where $f^n(w) =  f\circ f \cdots \circ f(w)$   are the $n$-fold iterates of a complex function $f$ at $w \in {\mathbb{C}}$, where $\mathbb{C}$ is the complex plane.
\par Iteration theory of holomorphic functions is mainly developed by P. Fatou and G. Julia. The dichotomy of the extended complex plane $\widehat{\mathbb{C}}$ in complex dynamical systems comes from the fact that some orbits behave very controlled or stable, while others exhibit chaotic or unstable dynamics.  The stable region is now called the Fatou set, whereas the chaotic region is called the Julia set. \\
\textbf{Notations :}  Throughout this article,
the set of all integers is  denoted by $\mathbb{Z}$.  Further, the real line is denoted by $\mathbb{R}$. Given a non-empty set $U\subset \mathbb{C}$, its closure and boundary is denoted by  $clos({U})$ and $\partial{U}$, respectively. Also, the disc of radius $r$ having centre $a\in \mathbb{C}$ is denoted by $\mathbb{D}(a,r)$. 
\par Let $f : \mathbb{C} \to \mathbb{\widehat{C}}$ be a transcendental meromorphic function such that $\infty$ is the only essential singularity.
The Julia set of $f$ is defined as the set of points $z \in \widehat{\mathbb{C}}$ such that either the family of iterates $\{f^n \}_{n>0}$ is not well-defined at $z$ or it is not a normal family in any neighbourhood of $z$ where it is well-defined.  It is denoted by $\mathcal{J}(f)$.  The Fatou set of $f$ is the complement of  $\mathcal{J}(f)$ in $\mathbb{\widehat{C}}$. It is denoted by $\mathcal{F}(f)$. A component of the Fatou set $\mathcal{F}(f)$ is known as a $Fatou$ $component$.  A Fatou component $U$ of $f$ is called $p$-$periodic$ if either $f(U) \subseteq U$, or  $f^p(U) \subseteq U$ but  $f^k(U) \not \subseteq U$ for $1\leq k< p$. Further, $U$  is said to be $invariant$ if $p=1$. If an invariant Fatou component $U$ of $f$  satisfies $f^{-1}(U) \subseteq U$, then $U$ is known as a completely invariant Fatou component. For details, see \cite{Milnor2006}.  
\par Let $U_r$ be a component of the preimage $f^{-1}(\mathbb{D}(a, r))$ for $r>0$ and $a \in \mathbb{C}$, and it satisfies if $r_1 <r_2$ implies $U_{r_1} \subset U_{r_2}$. Then there are two possibilities:\\
(a) $\bigcap \limits_{r>0} U_r= \{z\}$  for some $z \in \mathbb{C}$, or\\
(b) $\bigcap \limits_{r>0} U_r= \emptyset$.\\
In case (a),  $f(z)=a$. If  $f'(z) \neq 0$, then $z$ is called an ordinary point. If $f'(z)=0$, then $z$ is called a critical point and $a$ is called a critical value. It says  $f^{-1}$ has an algebraic singularity over $a$. The set of all critical values of $f$ is denoted by $CV(f)$.\\
In case (b), there exists a continuous curve $\alpha : [0, \infty) \to \widehat{\mathbb{C}}$ such that $\lim \limits_{t \to \infty} \alpha(t) = \infty$ and $\lim \limits_{t \to \infty} f(\alpha(t)) = a$, then $a$ is called an asymptotic value of $f$ and $\alpha(t)$ is called an asymptotic curve of $f$. It says $r \mapsto U_r$ defines a transcendental singularity of $f^{-1}$ over $a$. The set of all asymptotic values of $f$ is denoted by $AV(f)$.
A transcendental singularity of $f^{-1}$ over $a$ is called $direct$ if for some $r>0$ we have $f(z) \neq a$ for $z \in U_r$. If no such $r$ exists, then it is called $indirect$, see \cite{iv}. The $set$ $of$ $singularities$ $of$ $f^{-1}$ is defined as
$$ S(f):= clos({CV(f) \cup AV(f)}).$$
The $post$ $singular$ $set$ $of$ $f$ consists of the forward orbits of all points in $S(f)$  and their accumulation points, except the point infinity. It is denoted by $\mathcal{P}(f)$. If the Euclidean distance between $\mathcal{P}(f)$ and $\mathcal{J}(f) \setminus\{\infty\}$ is positive, then $f$ is called $topologically$ $hyperbolic$  \cite{bn}. A point $z_0 \in \mathbb{C}$ is an omitted value of a function $f$ if $z_0 \notin f(\mathbb{C})$.  
 \par  The class of transcendental meromorphic functions $f : \mathbb{C} \to \widehat{\mathbb{C}}$  such
that $f$ has either at least two poles or exactly one pole that is not an omitted value is known as $general$ $transcendental$ $meromorphic$ $function$.   
 The class of general transcendental meromorphic function with a bounded set of singularities of $f^{-1}$ is denoted by $\mathcal{B}$. It is an open question that there are at most two completely invariant Fatou components for a transcendental meromorphic function \cite{Ber93a}. Cao and Wang \cite{Cao} have provided an affirmative answer for meromorphic functions having the set of singularities is finite. Nayak and Zheng \cite{tn} have provided an affirmative answer for general meromorphic functions with at least one omitted value, except  where the function is of infinite order, has a single omitted value and no critical value. In this article, we consider a one-parameter family of functions in $\mathcal{B}$ with no omitted values but infinitely many critical values. It is shown that this family has a simply connected completely invariant Fatou component. Also, for a certain range of parameters, this family has exactly one completely invariant Fatou component and it is equal to $\mathcal{F}(f)$, but it does not contain all the singularities of $f^{-1}$. Moreover, the dynamics of this family are rich and complex.   
\subsection{Structure of the article}
 We consider a  one-parameter family of functions  $$ \mathbb{S} = \Big\{ \displaystyle f_{\lambda}(z) = \frac {\sin{z}}{z^2 + \lambda} \; \text{for} \;  z \in \mathbb{C}: \lambda >0 \Big\}.$$
This article aims to explore the dynamics of $f_{\lambda} \in \mathbb{S}$. The function $f_{\lambda}$ is non-periodic, but it shows some symmetric properties. Also, in Section \ref{2.0}, it will be shown that $f_{\lambda} \in \mathcal{B}$ and it has no omitted values. Further, by Part 2 of Remark \ref{rem9}, the inverse of $f_{\lambda}$ has exactly two indirect singularities over the finite asymptotic value, but no direct singularities. While exploring the real dynamics in Section \ref{2.1},  it is shown that $f_{\lambda}$ is chaotic on some interval of the real line. Further, there are different types of bifurcations that occur for $f_{\lambda}$ at certain values of $\lambda$ (see Sections \ref{2.1} and \ref{2.2}).  In Theorem \ref{t17}, it is proved that for $f_{\lambda}$, there are $\lambda_1$ and $\lambda_2$ (the existence of these two parameters is shown in Section \ref{2.2})  such that the Fatou set of $f_{\lambda}$ is equals to the union of basins of attraction when $\lambda \in (\lambda_1,\lambda_2) \cup (\lambda_2, \infty)$ (see Section $\ref{2.3}$). As a consequence of Theorem \ref{t17}, $f_{\lambda}$ is topologically hyperbolic for $\lambda \in (\lambda_1,\lambda_2)$ (see Remark \ref{r3}). Also, in Theorem \ref{t21}, the characterization of the Julia set $\mathcal{J}(f_{\lambda})$ is shown for $\lambda \in (\lambda_1, \infty)\setminus \{\lambda_2\}$. Further, Theorems \ref{t18}, \ref{tt19} and \ref{t19} describe some topological aspects of the Fatou components of $f_{\lambda}$ for a certain range of $\lambda$ values.
  It is observed that $\mathcal{F}(f_{\lambda})$ has a completely invariant Fatou component containing all indirect singularities of $f^{-1}_{\lambda}$  but not all critical values for $\lambda \geq \lambda_1$. Also, the Julia set $\mathcal{J}(f_{\lambda})$ is connected in $\widehat{\mathbb{C}}$ for $\lambda \geq \lambda_1$. Moreover, $\mathcal{F}(f_{\lambda})$ has a unique completely invariant Fatou component, say $U_{\lambda}$, such that $\mathcal{F}(f_{\lambda}) = U_{\lambda}$ for $\lambda > \lambda_2$.  
\section{Preliminaries}
\subsection{Properties of $f_{\lambda}$} \label{2.0}
The function $f_{\lambda}(z) = \frac{\sin {z}}{(z^2 + \lambda)}$ has only two non-omitted poles. Observe that the function $f_{\lambda}$ is non-periodic and of order one.   
\par The following proposition describes the symmetric properties of the Fatou set $\mathcal{F}(f_{\lambda})$ for $\lambda>0$.
\begin{prop}\label{p1}
Let $\displaystyle f_{\lambda} \in \mathbb{S}$. If $z \in \mathcal{F}(f_{\lambda})$ then   $\{-z, \overline{z}\} \subseteq \mathcal{F}(f_{\lambda})$. 
\end{prop}
\begin{proof}
 If $z \in \mathcal{F}(f_{\lambda})$, then the sequence $\{f^n_{\lambda}\}_{n>0}$ is normal in some neighbourhood of $\bar{z}$ (or $-z$), because $\{f^n_{\lambda}\}_{n>0}$ is normal in some neighbourhood of $z$, and $f^n_{\lambda}(\bar{z})= \overline{f^n_{\lambda}(z)}$ and $f^{n}_{\lambda}({-z})= -f^{n}_{\lambda}({z})$ for all $n \in \mathbb{N}$. Therefore, $\{\bar{z}, -z\} \subset \mathcal{F}(f_{\lambda})$.
\end{proof}
\begin{rem}
Using Proposition \ref{p1}, for $z \in \mathcal{J}(f_{\lambda})$, the set $\{-z, \overline{z}\} \subseteq \mathcal{J}(f_{\lambda})$. 
\end{rem}
\begin{lem}\label{ll1}
Let $f_{\lambda} \in \mathbb{S}$. Every critical point lies on either the real axis or the imaginary axis.
\end{lem}
\begin{proof}
If $z$ is a critical point of $f_{\lambda}$, then $f'_{\lambda}(z)=0$. Further, it gives
\begin{equation*}
       \cot{z}= \frac{2z}{z^2 + \lambda}.
       \label{1}
\end{equation*} 
   If $x\neq 0$ and $y \neq 0$, then 
\begin{equation*}  \frac{\cos{(x+iy)} }{\sin{(x+iy)}} = \frac{2(x+iy)}{(\lambda +x^2 -y^2) + i2xy}.
  \label{2}
\end{equation*}
 Further, we have
\begin{equation*} 
   \frac{\sin{2x} + i \sinh{2y}}{\cos{2x}- \cosh{2y}} = \frac{-2x(\lambda +x^2 +y^2) 
   +i2y(\lambda -x^2-y^2)}{(\lambda +x^2 -y^2)^2 +(2xy)^2}.\label{3}
\end{equation*}
By comparing real and imaginary parts, we have
\begin{equation} 
   \frac{\sin{2x} }{\cos{2x}- \cosh{2y}} = \frac{-2x(\lambda +x^2 +y^2) }{(\lambda +x^2 -y^2)^2 +(2xy)^2},\label{4}
\end{equation}
   and
\begin{equation} 
   \frac{ \sinh{2y}}{\cos{2x}- \cosh{2y}} = \frac{2y(\lambda -x^2-y^2)}{(\lambda +x^2 -y^2)^2 +(2xy)^2}.
   \label{5}
\end{equation}
After dividing Equation (\ref{4}) by Equation (\ref{5}), we get
\[\frac{(\sin{2x})/2x}{(\sinh{2y})/2y}= \frac{x^2 +y^2 + \lambda}{x^2 +y^2 - \lambda}.\] 
Since 
   \[\left|\frac{\sin{2x}/2x}{\sinh{2y}/2y} \right|<1\] and \[\left|\frac{x^2 +y^2 + \lambda}{x^2 +y^2 - \lambda} \right|>1\] for $\lambda>0$,  
    $f_{\lambda}(x+iy)$ has no critical points for $x \neq 0$  and $y \neq 0$.  Hence, the result holds.  
\end{proof}
Using Lemma \ref{ll1}, the following proposition discusses the singularities of $f^{-1}_{\lambda}$. Further, it infers that $f_{\lambda} \in \mathcal{B}$ for $\lambda>0$. 
\begin{prop}\label{p2}
	Let $\displaystyle f_{\lambda} \in \mathbb{S}$. The set of singularities of $f^{-1}_{\lambda}$ is bounded.  Further, the set $CV(f_{\lambda})$ is not finite.
\end{prop}
\begin{proof} 
	 On the real axis, we have
	\[ f'_{\lambda}(x) = \displaystyle \frac{(x^2 +\lambda)\cos{x} -2x \sin{x}}{ (x^2 +\lambda)^2}.\]
    Let $N(x)$ represent the numerator of the function $f'_{\lambda}(x)$. Now,  $N'(x)= (-\sin{x})(x^2 +\lambda +2)$ and  
\begin{equation*}
				N'(x) 
\begin{cases}
                       =0 &  \text{at} \; x =2n\pi, \\
                       <0 &   \text{for} \; x \in ( 2n\pi, (2n+1)\pi),\\
                       =0 &  \text{at} \; x =(2n+1)\pi, \\
					>0 &  \text{for} \; x \in ( (2n+1)\pi, (2n+2)\pi),   
\end{cases}
\end{equation*}
where $n \in \mathbb{Z}$. Also, note that   $N(2n \pi)> 0$ and $N((2n+1) \pi)< 0$. 
Therefore, there exists $p_{\lambda, n} \in (n\pi, (n+1)\pi)$ such that
\begin{equation}
		f'_{\lambda}(x) 
		\begin{cases}
			>0 &  \text{for} \; x \in [ 2n\pi, p_{\lambda, 2n}),   \\
			=0 &  \text{at} \; x=p_{\lambda, 2n}, \\
			<0 &   \text{for} \; x \in (p_{\lambda, 2n}, (2n+1)\pi],
		\end{cases}
		\label{g1}
\end{equation} 
\begin{align*}
		f'_{\lambda}(x) 
		\begin{cases}
			<0 &  \text{for} \; x \in [(2n+1)\pi,p_{\lambda, 2n +1}),   \\
			=0 &  \text{at} \; x=p_{\lambda, 2n +1}, \\
			>0 &   \text{for} \; x \in (p_{\lambda, 2n +1}, (2n+2)\pi],
		\end{cases}
\end{align*}
where $n \in \mathbb{Z}$. It implies $f_{\lambda}(x)$ has a critical point $p_{\lambda,n}$ in the interval $ (n\pi, (n+1)\pi)$ for each $n \in \mathbb{Z}$. Also, the set $\left\{ \frac{\cos{p_{\lambda,n}}}{2p_{\lambda,n}} : n \in \mathbb{Z} \right\} \subseteq CV(f_{\lambda}(x))$ and it is bounded.\\ 
 On the imaginary axis,  $f_{\lambda}(iy) = i\frac{\sinh{y}}{ \lambda -y^2}$ for $y \in \mathbb{R}\setminus \{-\sqrt{\lambda},\sqrt{\lambda}\}$.  
 Let $h_{\lambda}(y)=-if_{\lambda}(iy)$.  We have $h'_{\lambda}(y) =\displaystyle \frac{(\lambda -y^2)\cosh{y} +2y \sinh{y}}{(\lambda -y^2)^2}$. Observe that $h'_{\lambda}(y)>0$ when $y \in (0, \sqrt{\lambda})$. Let $N(y)$ represent the numerator of the function $h'_{\lambda}(y)$. Now, $N(\sqrt{\lambda})>0$ and $N'(y)= (\lambda +2 -y^2)\sinh{y}$. So, $N'(y)>0$ for $y \in (\sqrt{\lambda},\sqrt{\lambda +2})$,  $N'(\sqrt{\lambda +2})=0$ and  $N'(y)<0$ for $y \in (\sqrt{\lambda +2}, \infty)$.
 Therefore, there exists  $c_{\lambda}>0$ such that
\begin{equation}
		h'_{\lambda}(y) 
		\begin{cases}
			<0 & \text{for} \; y  \in (-\infty, -c_{\lambda} ),\\
			=0 & \text{for} \; y=-c_{\lambda}, \\
			>0 &  \text{for} \; y \in (-c_{\lambda}, c_{\lambda})\setminus \{-\sqrt{\lambda},\sqrt{\lambda}\},   \\
			=0 &  \text{for} \; y=c_{\lambda}, \\
			<0 &   \text{for} \; y \in (c_{\lambda}, \infty),
		\end{cases}
		\label{g2}
\end{equation}
since $h'_{\lambda}(y)$ is even.	
 Therefore, the set $\left \{ -\frac{\cosh{c_{\lambda}}}{2c_{\lambda}}, \frac{\cosh{c_{\lambda}}}{2c_{\lambda}} \right\} \subset CV(h_{\lambda}(y))$.\\
 Observe that the real and imaginary axes are invariant under $f_{\lambda}$, i.e., $f_{\lambda}(\mathbb{R}) \subseteq \mathbb{R}$ and $f_{\lambda}(i\mathbb{R}) \subseteq i\mathbb{R}$. Consequently, by Lemma \ref{ll1}, $$CV(f_{\lambda})= \left\{ \frac{\cos{p_{\lambda,n}}}{2p_{\lambda,n}} : n \in \mathbb{Z} \right\}  \bigcup\left \{ -i\frac{\cosh{c_{\lambda}}}{2c_{\lambda}}, i\frac{\cosh{c_{\lambda}}}{2c_{\lambda}} \right\}.$$
 Since $f_{\lambda}(x) \to 0$ as $x \to \infty$, $0 \in AV(f_{\lambda})$.
 Suppose $l$ is a non-zero finite asymptotic value of $f_{\lambda}$. So, there exists an asymptotic curve $\gamma(t)$ such that $f_{\lambda}(\gamma(t)) \to l$ as $t \to \infty$. Also, $|f_{\lambda}(\gamma(t))| \to |l|$ as $t \to \infty$. Let $\gamma(t)= \gamma_1(t)+i \gamma_2(t)$. Then $|\sin(\gamma(t))| = \sqrt{\sin^2(\gamma_1(t)) + \sinh^2(\gamma_2(t))}$. Also, it is easy to see that $ |\gamma^{2}(t)+\lambda| \to \infty$ as $t \to \infty$. Observe that if $|\sin(\gamma(t))|$ is bounded as $t \to \infty$, then $l=0$, therefore, it contradicts $l \neq 0$. Also, if $|\sin(\gamma(t))| \to \infty$ as $t \to \infty$, then $l = \infty$, therefore, it contradicts $l$ is finite. Consequently, $$AV(f_{\lambda})= \{0\}.$$ Hence,  the result holds.   
 \end{proof}
 \begin{cor}
 Let $ f_{\lambda} \in \mathbb{S}$. There are no omitted values of $f_{\lambda}$. 
\end{cor}

\begin{prop}
Let $ f_{\lambda} \in  \mathbb{S}$. The inverse of $f_{\lambda}$ has two indirect singularities over $0$. 
\end{prop}
\begin{proof}
Using Proposition \ref{p2}, $0$ is a transcendental singularity of $f^{-1}_{\lambda}$. Also, $f_{\lambda}(x) \to 0$ as $|x| \to \infty$. Further, the set  $f^{-1}_{\lambda}(0)=  \{ n \pi : n \in \mathbb{Z}\}$. Therefore, there exist $U^{-}(r)$ and $U^{+}(r)$  components of $f_{\lambda}^{-1}(\mathbb{D}(0, r))$ such that they contain a tail of negative real axis and positive real axis, respectively. Consequently, there exists no $r$ such that $f(z)\neq 0$ for $z \in U^{-}(r) ( \text{or} \; U^{+}(r))$. Therefore, the result holds.
\end{proof}
\subsection{Dynamics of $f_{\lambda}(x)$ and bifurcations}	\label{2.1}
This section explores the real dynamics of the function $f_{\lambda} \in \mathbb{S}$.
\par	The function $f_{\lambda}(x)$ is odd and $f_{\lambda}(x) \rightarrow 0$ as $|x| \rightarrow \infty$. Also,
 \begin{equation*}
f_{\lambda}(x) 
\begin{cases}
	=0 &  \text{at} \; x =2n\pi, \\
	>0 &  \text{for} \;  x \in (2n\pi, (2n+1)\pi),   \\
	=0 &  \text{at} \; x =(2n+1)\pi, \\
	<0 &   \text{for} \; x \in ((2n+1)\pi,(2n+2)\pi),
\end{cases}
\label{e1}
\end{equation*}
where $n \in \mathbb{Z}$.
We have
$$ \displaystyle f''_{\lambda}(x)=\displaystyle \frac{(\sin{x})(8x^2 - (x^2 + \lambda)^2 -2(x^2 + \lambda)) - 4x(x^2 + \lambda) \cos{x}}{ (x^2 +\lambda)^3}$$
for $x \in \mathbb{R}$. Let $N(x)$ represent the numerator of the function $f''_{\lambda}(x)$. 
Then we have 
$$N'(x) = 12x \sin{x} - (x^2 + \lambda)(x^2 + \lambda +6)(\cos{x})$$ and $$N''(x) = ((x^2 + \lambda +6)\sin{x} -4x \cos{x})(x^2 + \lambda) + 12\sin{x}.$$ Since $(x^2 + \lambda +2)\cos{x} +6x \sin{x} >0$ for $x \in (0, \frac{\pi}{2})$, $N''(x)>0$ for $x \in (0, \frac{\pi}{2})$. Note that $N'(0)<0$, $N'(\frac{\pi}{2})>0$, and $N'(x)>0$ for $x \in (\frac{\pi}{2}, \pi)$. Therefore, there exists $n_{\lambda}>0$ such that
\begin{equation}
f''_{\lambda}(x) 
\begin{cases}
	<0 &  \text{for} \; x \in ( 0,n_{\lambda}),   \\
	=0 &  \text{at} \; x=n_{\lambda}, \\
	>0 &   \text{for} \; x \in (n_{\lambda}, \pi),
\end{cases}
\label{e1}
\end{equation}
because $f''_{\lambda}(0)= 0$ and $f''_{\lambda}(\pi)>0$.
\par Consider $ \Psi(x) = \frac{\sin x}{x} -x^2$ for $x \in (0,\infty)$, and clearly, $\Psi(0)=1$. Now, $\displaystyle \Psi'(x)= \frac{x \cos x - \sin x -2x^3}{x^2}$ and $\Psi'(0)=0$. Let $N(x)$ represent the numerator of the function $ \Psi'(x)$.
Then  $N'(x)= -x^2( \frac{\sin{x}}{x}+6) <0$ for $x \in (0,\infty)$.  Therefore, $\Psi'(x)<0$ and hence, $\Psi(x)$  is strictly decreasing.\\ 
Let $x_{0} \in (0, 1)$ be the positive solution of $x \Psi(x)=0$.\\
 Consider $\phi(x)= \frac{\cos{x}}{(\Psi(x)+x^2)} - \frac{2x \sin{x}}{(\Psi(x)+x^2)^2}$ for $x \in (0, x_{0})$. Observe that $ \phi(0)=1$.
Now, $\phi'(x)$ equals to 
  $$\frac{ -(x^2  +\Psi(x))((4x+ \Psi'(x))(\cos{x})+ (\sin{x})(\Psi(x)  +x^2  +2)) +(2x\sin{x})(2x + \Psi'(x))}{(\Psi(x)+x^2)^3}.$$ 
  Since $x(6x- \sin{x})>0$ on $(0, \frac{\pi}{2})$ and it implies $x\cos{x}- \sin{x}+2x^3>0$. Therfore,  $ 4x+\Psi'(x) >0$.  Also, $x\cos{x}-\sin{x}<0 $ on $(0, x_0)$     implies that $2x +\Psi'(x)<0$. Consequently, $\phi'(x) <0$ for $x \in (0, x_{0})$ and hence, $\phi(x)$ is strictly decreasing. Note that $\phi(x_{0})<-1$. Consequently, there exists $x^* \in (0, x_{0})$ such that
\begin{equation}
  	\phi(x)  
  	\begin{cases}
  		>-1, \; \text{but} \, <0 &  \text{for} \; x \in ( 0,x^*),   \\
  		=-1 &  \text{at} \; x=x^*, \\
  		<-1 &   \text{for} \; x \in (x^*, x_{0}).
  	\end{cases}
  	\label{e3}
  \end{equation}
Set $\lambda^* = \Psi(x^*)=(x^*)^2 -\cos{x^*}$. %The approximate value of $\lambda^*$ is $0.117$.
\par %In the following propositions, we describe the dynamics of $f_{\lambda}(x)$. 
The following proposition examines the multiplier of fixed points of $f_{\lambda}$.  
\begin{prop} \label{t1}
Let $ \displaystyle f_{\lambda}(x) = \frac {\sin{x}}{x^2 + \lambda}$ for $\lambda>0$.  There exists $\lambda^* (\approx 0.117)$ as follows:  
\begin{enumerate}[label=(\alph*)]
\item If $0 < \lambda <\lambda^*$, then $f_{\lambda}(x)$ has two non-zero repulsive fixed points, and $0$ is a repulsive fixed point.
\item If $\lambda = \lambda^*$, then  $f_{\lambda}(x)$ has two non-zero rationally indifferent fixed points $-x^*$ and $x^*$, and $0$ is a repulsive fixed point.
\item If $\lambda^* < \lambda < 1$, then $f_{\lambda}(x)$ has two non-zero attractive fixed points $-x_{\lambda}$ and $x_{\lambda}$, and $0$ is a repulsive fixed point.
\item If $\lambda=1$, then  $0$ is the rationally indifferent fixed point of  $f_{\lambda}(x)$.
\item If $\lambda>1$, then  $0$ is the attractive fixed point of  $f_{\lambda}(x)$.
\end{enumerate}
\end{prop}
\begin{proof}
Define $g_{\lambda}(x)= f_{\lambda}(x)-x$ for $x \in \mathbb{R}$. Observe that $g_{\lambda}(x)=0$ if and only if $f_{\lambda}(x)=x$. Note that $g'_{\lambda}(x)= f'_{\lambda}(x)-1$ and $g''_{\lambda}(x)= f''_{\lambda}(x)$. 
 Further,
\begin{enumerate}
\item[(i)]  $|g_{\lambda}(x)|>0$ for $|x|\geq \pi$;
\item[(ii)] $g_{\lambda}(x)$ has unique local maximum at $\tilde{x}= \tilde{x}(\lambda) \in (0,\pi)$ and it has  unique local minimum at $-\tilde{x}$ for $\lambda \in (0,1)$;
\item[(iii)] $|g_{\lambda}(x)|>0$ on $0< |x|<\pi$ for $\lambda \geq1$.
\end{enumerate}
\par Since $|f_{\lambda}(x)|<1$ when $|x|\geq \pi$, it follows that $|g_{\lambda}(x)|>0$ for $|x|\geq \pi$ and thus, (i) holds. Note that $g'_{\lambda}(0) (= \frac{1}{\lambda}-1)>0$ for $\lambda \in (0,1)$ and $g'_{\lambda}(\pi) (=\frac{-1}{(\pi^2 + \lambda)}-1)<0$ for all $\lambda>0$.  So, by (\ref{e1}) and continuity of $g'_{\lambda}(x)$, there exists a unique $\tilde{x}=\tilde{x}(\lambda)>0$ such that $g'_{\lambda}(\tilde{x})=0$, $g'_{\lambda}(x)>0$ when $x \in (0, \tilde{x})$ and  $g'_{\lambda}(x)<0$ when $x \in (\tilde{x},\pi)$ for $\lambda \in (0,1)$. Observe that $g_{\lambda}(x)$ is odd. Thus, $g_{\lambda}(x)$ attains local maximum value at $\tilde{x}$ and it attains local minimum value at $-\tilde{x}$ for $\lambda \in (0,1)$. Therefore, (ii) holds. Since $|f_{\lambda}(x)|<1$ for $0< |x|<\pi$ when $\lambda \geq1$, (iii) holds. 
\par  Note that $g_{\lambda}(0)=0$ and $g_{\lambda}(\pi)<0$. Also, $g_{\lambda}(x) \to -\infty$ as $x \to \infty$ and $g_{\lambda}(x) \to \infty$ as $x \to -\infty$.  Thus by (i) and (ii), $g_{\lambda}(x)$ has three zeros 
for $\lambda < 1$.
Let $x_{\lambda}>0$ be such that $g_{\lambda}(x_{\lambda})=0$, it gives $\lambda= \Psi(x_{\lambda})$.\\
(a) If $0<\lambda<\lambda^*$, then $\Psi(x_{\lambda}) <\Psi(x^*)$. Since $\Psi(x)$ is strictly decreasing on $(0, x_{0})$, we have $x^*<x_{\lambda}$. Thus by (\ref{e3}), $g'_{\lambda}(x_{\lambda})= \phi(x_{\lambda})-1<-2$. Consequently, $f'_{\lambda}(x_{\lambda})<-1$. Also, $g'_{\lambda}(0)>0$. Therefore, $-x_{\lambda}$, $0$ and $x_{\lambda}$  are repulsive fixed points of $f_{\lambda}(x)$.\\
(b) If $\lambda=\lambda^*$, then $\Psi(x_{\lambda}) =\Psi(x^*)$.  The injectivity of $\Psi(x)$ on $(0, x_{0})$ implies that $x_{\lambda}=x^*$. Thus  by (\ref{e3}), it follows that $g'_{\lambda^*}(x^*)= \phi(x^*)-1 =-2$. Consequently, $f'_{\lambda^*}(x^*) = -1$. Also, $g'_{\lambda}(0)>0$.
Therefore, $-x^*$ and $x^*$ both are rationally indifferent fixed points, and $0$ is a repulsive fixed point  of $f_{\lambda}(x)$ (see in Figure~\ref{rr}(a)).\\
(c) If $\lambda^* <\lambda<1$, then $ \Psi(x^*)< \Psi(x_{\lambda})$. Since $\Psi(x)$ is strictly decreasing on $(0, x_{0})$, we have $x^* >x_{\lambda}>0$. Thus by (\ref{e3}), $-2< g'_{\lambda}(x_{\lambda}) = \phi(x_{\lambda}) -1<0$. Consequently, $-1< f'_{\lambda}(x_{\lambda})<1$. Also, $g'_{\lambda}(0)>0$. Therefore, $-x_{\lambda}$ and $x_{\lambda}$ both are attractive fixed points of $f_{\lambda}(x)$, and $0$ is a repulsive fixed point.\\
(d) If $\lambda=1$, then $g'_{\lambda}(0)=0$. Thus $f'_{\lambda}(0)=1$. Therefore, by (i) and (iii), $0$ is the rationally indifferent fixed point of $f_{\lambda}(x)$ (see in Figure~\ref{rr}(b)).\\
(e) If $\lambda>1$, then $g'_{\lambda}(0)<0$. Thus $f'_{\lambda}(0)<1$. Therefore, by (i) and (iii), $0$ is the attractive fixed point of $f_{\lambda}(x)$ (see in Figure~\ref{rr}(c)).
\end{proof}
\begin{figure}[h!]
	\centering
	\subfloat[]{\includegraphics[width=0.45\textwidth]{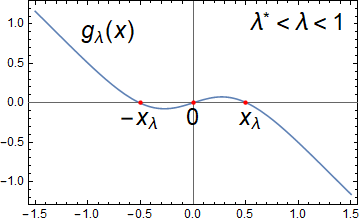}} \hspace{1cm}
	\subfloat[]{\includegraphics[width=0.45\textwidth]{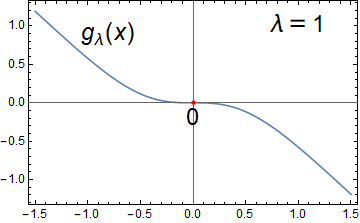}} \\
	\subfloat[]{\includegraphics[width=0.45\textwidth]{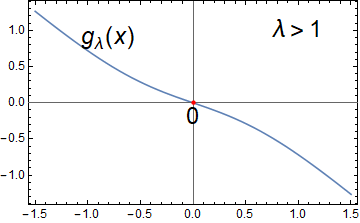}}
	\caption{The graphs of $g_{\lambda}(x)$ for \text{(a)} $\lambda^*< \lambda < 1, \text{(b)} \; \lambda= 1$ and $\text{(c)} \;\lambda > 1$.}
\label{rr}
\end{figure}
\par  Let $f$ be a continuous function on an interval $I$ of the real line to itself. The function $f$ is called $turbulent$ if there exist compact subintervals $J$ and $K$ of $I$ such that\\
(1) the intersection $J \cap K$ contains at most one point, and\\
(2) the union $J \cup K \subseteq f(J) \cap f(K)$.\\
If some iterate of $f$ is turbulent, then $f$ on $I$ is said to be  $chaotic$ \cite{block2006dynamics}. \\
Consider $\psi(x)=2x\tan{x}-x^2$ for $x \in (0, \frac{\pi}{2})$. Now, $\psi'(x)= 2(\tan{x})(1+\tan{x})>0$ and hence, $\psi(x)$ is strictly increasing on  $(0, \frac{\pi}{2})$.
\par In the following proposition, it is proved that $f_{\lambda}(x)$ is chaotic on some interval for a certain range of  $\lambda$  values. Let $p_{\lambda^{**}}$ be the positive solution of the equation $\displaystyle \frac{\cos x}{2x} = \pi$, and   $ \lambda^{**} = \frac{\sin p_{\lambda^{**}}}{\pi} - \Big(\frac{\cos p_{\lambda^{**}}}{2\pi} \Big)^2$.
\begin{prop} \label{t7}
Let $\displaystyle f_{\lambda}(x) = \frac{\sin{x}}{x^2 + \lambda}$ for $\lambda >0$. There exists $\lambda^{**} (\approx 0.0251)$ such that  if $0 < \lambda \leq \lambda^{**}$, then $f_{\lambda}(x)$ is chaotic on $[-f_{\lambda}(p_{\lambda}), f_{\lambda}(p_{\lambda})]$, where $p_{\lambda}= p_{\lambda,0}$ is a critical point of $f_{\lambda}(x)$.  
\end{prop}
\begin{proof}
If $0< \lambda_1< \lambda_2<1$, then $\psi(p_{\lambda_1})< \psi(p_{\lambda_2})$. Since $\psi(x)$ is strictly increasing, $0<p_{\lambda_1}< p_{\lambda_2}<1$. Note that $\frac{\cos{x}}{2x}$ is strictly decreasing on $\left(0,\frac{\pi}{2}\right)$. So,  $\frac{\cos{p_{\lambda_1}}}{2p_{\lambda_1}} > \frac{\cos{p_{\lambda_2}}}{2p_{\lambda_2}} > \frac{\cos{1}}{2} > \frac{1}{2\pi}$. Therefore, $\left| \frac{\cos{p_{\lambda}}}{2p_{\lambda}} \right| > \left| \frac{\cos{p_{\lambda,j}}}{2p_{\lambda,j}} \right|$ for $j \in \mathbb{Z}\setminus \{-1,0\}$  and $\lambda \in (0,1)$, since $p_{\lambda,j} \in (j\pi, (j+1)\pi)$. It implies $-\frac{\cos{p_{\lambda}}}{2p_{\lambda}} $ and  $\frac{\cos{p_{\lambda}}}{2p_{\lambda}} $ are minimum and maximum values of $f_{\lambda}(x)$, respectively. Observe that $\frac{\cos{p_{\lambda}}}{2p_{\lambda}} $ increases as $\lambda$ decreases such that $\frac{\cos{p_{\lambda}}}{2p_{\lambda}} \to \infty$ as $\lambda \to 0$. Notice that $\frac{\cos{p_{\lambda}}}{2p_{\lambda}} <1$ at $\lambda=1$. Therefore, there exists a unique $p_{\lambda^{**}}$ such that $ \frac{\cos p_{\lambda^{**}}}{2p_{\lambda^{**}}} = \pi$, and $\displaystyle \lambda^{**} = \frac{\sin p_{\lambda^{**}}}{\pi} - \Big(\frac{\cos p_{\lambda^{**}}}{2\pi}\Big)^2$. Further, if $\lambda \leq \lambda^{**}$, then $f_{\lambda}(p_{\lambda}) \geq \pi$. Take $I= [-f_{\lambda}(p_{\lambda}), f_{\lambda}(p_{\lambda})]$, $J= [0, p_{\lambda}]$ and $K=[p_{\lambda}, \pi]$. Then $J \cup K \subset f_{\lambda}(J)  \cap f_{\lambda}(K)$ for $\lambda \leq \lambda^{**}$. Hence, by definition, the result holds. 	
\end{proof}
\begin{rem} \label{r2}
Using Proposition \ref{t7}, $f_{\lambda}(x)$ is unimodal on $[0,\pi]$ for $\lambda \in [\lambda^{**}, \infty)$, that means, $f_{\lambda}(x)$ maps $[0,\pi]$ to itself and it has a unique critical point $p_{\lambda} \in (0,\pi)$.
\end{rem}
The following lemma describes the relation between $\lambda^*$ and $\lambda^{**}$, where $\lambda^*$ and $\lambda^{**}$ are defined in Proposition \ref{t1} and Proposition \ref{t7}, respectively. 
\begin{lem}\label{l5}
Let $\displaystyle f_{\lambda}(x) = \frac{\sin{x}}{x^2 + \lambda}$ for $\lambda >0$. Then $\lambda^{**}\leq \lambda^*$. 
\end{lem}
\begin{proof}
By Lemma \ref{t7}, $f_{\lambda^{**}}(x)$ is chaotic on the invariant set $[0,\pi]$. Thus, by Theorem 4.1 of \cite{Bernd}, $f_{\lambda^{**}}(x)$ is topologically transitive on $[0,\pi]$. Therefore, $x_{\lambda^{**}}$ is not an attractive fixed point of $f_{\lambda^{**}}(x)$ and hence,  $\lambda^{**} \leq \lambda^*$ by Proposition \ref{t1}.
 \end{proof}
\par Let $S_{\lambda,1} := \{ x \in \mathbb{R} : f_{\lambda}(x) >0 \}$ and  $S_{\lambda,2} := \{ x \in \mathbb{R} : f_{\lambda}(x) <0 \}$. In the following proposition, it is shown that the bifurcation occurs at $\lambda=1$ for $f_{\lambda}(x)$.
\begin{prop} \label{t2}
Let $ f_{\lambda}(x) = \frac {\sin{x}}{x^2 + \lambda}$ for $\lambda >0$.
\begin{enumerate}[label=(\alph*)]
\item If $\lambda > 1$, then  $f^n_{\lambda}(x) \to 0$ as $n \to \infty$ for $x \in \mathbb{R}$.
\item If $\lambda = 1$, then $f^n_{\lambda}(x) \to 0$ as $n \to \infty$ for $x \in \mathbb{R}$.
\item If $\lambda^* <\lambda <1$, then  $f^n_{\lambda}(x) \to x_{\lambda}$ as $n \to \infty$ for $x \in S_{\lambda,1}$ and $f^n_{\lambda}(x) \to -x_{\lambda}$ as $n \to \infty$ for $x \in S_{\lambda,2}$, where $-x_{\lambda}$ and $x_{\lambda}$ are attractive fixed points of $f_{\lambda}$. 
\end{enumerate} 
\end{prop}
\begin{proof}
If $\displaystyle x \in \mathbb{R}\setminus \{0\}$, then  $|f_{\lambda}(x)| < |x|$ for $\lambda \geq 1$.\\
(a) The sequence $\{|f_{\lambda}^{n}(x)|\}_{n>0}$ is strictly decreasing and bounded below by $0$ and so, by Monotonic convergence theorem, it converges to $0$. Hence, the result holds.\\ 
(b) On parallel arguments to  Part (a) of Proposition \ref{t2}, the result holds.\\
(c) By Remark \ref{r2} and Lemma \ref{l5},  $f_{\lambda}(x)$ is  unimodal on $[0,\pi]$ for $\lambda > \lambda^*$.
Also, $\displaystyle \frac{\partial f_{\lambda}}{\partial \lambda} (x)= \frac{-\sin x}{(x^2 + \lambda)^2} <0$ when $x \in (0, \pi)$.  Thus by Proposition \ref{t1},   
there exists a unique $\hat{\lambda} \in (\lambda^*,1)$ such that the non-zero fixed point and critical point of $f_{\lambda}(x)$ has the following property: 
\begin{equation*}  
  	\begin{cases}
  		x_{\lambda}>p_{\lambda} &  \text{for} \; \lambda \in (\lambda^*,\hat{\lambda}),   \\
  		x_{\lambda}=p_{\lambda} &  \text{at} \; \lambda= \hat{\lambda}, \\
  		x_{\lambda}<p_{\lambda} &   \text{for} \; \lambda \in (\hat{\lambda},1),
  	\end{cases}
\end{equation*}
where $p_{\lambda}=p_{\lambda,0}$.\\
\underline{Case I:} For $\lambda \in [\hat{\lambda},1)$.\\
 We show that if $x \in (0, \pi)$, then $f^n_{\lambda}(x) \rightarrow x_{\lambda}$ as $ n \rightarrow \infty$.\\
Observe that the local maximum value $M_{\lambda}$ of $f_{\lambda}(x)$ in $[0,\pi]$ that is attained at the critical point $p_{\lambda}$ satisfies $M_{\lambda} < p_{\lambda}$. Therefore, $f_{\lambda}(x)$ maps $[p_{\lambda}, \pi]$ into $[0, M_{\lambda}] \subset [0, p_{\lambda}]$ and  
\begin{equation}
f_{\lambda}(x) -x 
\begin{cases}
	>0 & \text{for} \; x \in (0,x_{\lambda}),   \\
	<0 &  \text{for} \; x \in (x_{\lambda},\pi). 
\end{cases}
\label{eg1}
\end{equation}
Therefore, by  (\ref{eg1}),
\begin{equation}
0< f_{\lambda}(x)- x_{\lambda} < x- x_{\lambda}
\label{eg2}
\end{equation}
for $x_{\lambda} < x < p_{\lambda}$. Since $f'_{\lambda}(x) > f'_{\lambda}(x_{\lambda}) >0$, $f_{\lambda}(x)>x$   for $x \in (0,x_{\lambda})$. So, $0< x_{\lambda} -f_{\lambda}(x) < x_{\lambda} -x$.
Moreover, by  (\ref{eg2}),  $|f_{\lambda}(x) - x_{\lambda}| < |x-x_{\lambda}|$ for $x \in (x_{\lambda},p_{\lambda}]$. Consequently, for $x \in (0,p_{\lambda}]$ and $x \neq x_{\lambda}$, we have $|f_{\lambda}(x) - x_{\lambda}| < |x-x_{\lambda}|$. 
Thus  $f^{n}_{\lambda}(x) \rightarrow x_{\lambda}$ as $n \rightarrow \infty$ for $x \in (0,\pi)$. Note that $f_{\lambda}(x)$ has no 2-periodic point in the interval $(0,\pi)$  for $\lambda \in [\hat{\lambda},\infty)$.\\
\begin{figure}[h!]
	\centering
	\subfloat[]{\includegraphics[width=0.45\textwidth]{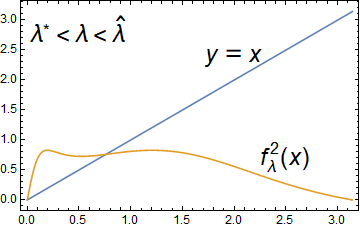}} \hspace{1cm}
	\subfloat[]{\includegraphics[width=0.45\textwidth]{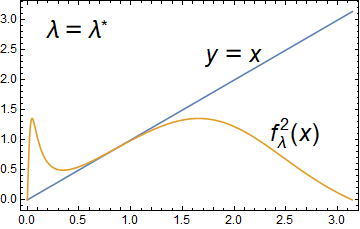}} \\
	\subfloat[]{\includegraphics[width=0.45\textwidth]{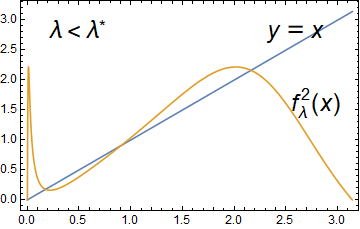}}
	\caption{The graphs of $f^2_{\lambda}(x)$ for \text{(a)} $\lambda^*< \lambda < \hat{\lambda}, \text{(b)} \; \lambda= \lambda^*$ and $\text{(c)} \;\lambda < \lambda^*$.}
	\label{fof}
\end{figure}
\underline{Case II:} For $\lambda \in (\lambda^*, \hat{\lambda})$.\\
\par \noindent
Claim: $f_{\lambda}$ has no 2-periodic point.\\
\par  Suppose not, there exists $\tilde{\lambda} \in (\lambda^*, \hat{\lambda})$ such that $f_{\tilde{\lambda}}(x)$ has a $2$-periodic point in $(0, \pi)$, say $x_{\tilde{\lambda},2} (< f_{\tilde{\lambda}}(x_{\tilde{\lambda},2})$). But $f_{\lambda}(x)$ has no $2$-periodic point for $\lambda > \tilde{\lambda}$. Therefore, by  Proposition \ref{t1}, $f^2_{\lambda}(x)$ has the only attractive  fixed point $x_{\lambda}$  satisfying $0< (f^2_{\lambda})'(x_{\lambda})<1$ for $\lambda > \tilde{\lambda}$. It implies that $f^2_{\lambda}(x)> x$ when $x \in (0,x_{\lambda})$ and $f^2_{\lambda}(x)< x$ when $x \in (x_{\lambda}, \pi)$ for $\lambda > \tilde{\lambda}$.
Therefore, $(f^2_{\tilde{\lambda}})'(x_{\tilde{\lambda},2})=1$ and it implies that $x_{\tilde{\lambda},2}$ is a rationally indifferent fixed point of $f^2_{\tilde{\lambda}}(x)$. Observe that $f_{\lambda}^2(x)$ has three critical points in $(0, \pi)$ for $\lambda \in (\lambda^*, \tilde{\lambda})$. Assume $0< p_{\lambda,1} < p_{\lambda} < p_{\lambda,2}$  as critical points of $f^2_{\lambda}(x)$. Notice that $f^2_{\lambda}(p_{\lambda,1}) (= f^2_{\lambda}(p_{\lambda,2})) > f^2_{\lambda}(p_{\lambda})$ since $f_{\lambda}(x)$ is unimodal on $[0,\pi]$ and $p_{\lambda}< x_{\lambda}$. Therefore, we have the following order $$0<p_{\tilde{\lambda}} \leq x_{\tilde{\lambda},2}< x_{\tilde{\lambda}}<f_{\tilde{\lambda}}(x_{\tilde{\lambda},2})\leq  p_{\tilde{\lambda},2}$$ (see in Figure \ref{fof}). Consequently, by Lemma 2.6 of \cite{Singer}, $(f^2_{\tilde{\lambda}})'(x_{\tilde{\lambda}})>1$. Hence, it contradicts $x_{\tilde{\lambda}}$ is an attractive fixed point of $f_{\tilde{\lambda}}(x)$ and it implies that our claim holds.   
\par By Sarkovskii's order theorem, $f_{\lambda}(x)$ has no periodic point of period greater than or equal to 2 inside $(0,\pi)$.
Also, notice that $f_{\lambda}([0,\pi]) \subset [0,\pi]$. Therefore, by Proposition \ref{t1}, $f^n_{\lambda}(x) \rightarrow x_{\lambda}$ when $x \in (0, \pi)$ and since $f_{\lambda}(x)$  is odd, so, $f^n_{\lambda}(x) \rightarrow -x_{\lambda}$  when $x \in (-\pi,0)$ as $n \rightarrow \infty$. 
\par In Proposition \ref{t7}, we showed that $f_{\lambda}(x)$  attains its maximum value at $ p_{\lambda} \in (0,\pi)$ and its minimum value at $\displaystyle -p_{\lambda}$. 
Therefore, $f_{\lambda}(S_{\lambda,1}) \subset (0,\pi)$ and $f_{\lambda}(S_{\lambda,1}) \subset (-\pi,0)$. Hence, by Case I and Case II, the result holds. 
\end{proof}
\begin{rem}
Using Proposition \ref{t1} and Proposition \ref{t2}, a pitchfork supercritical bifurcation occurs at $\lambda=1$ for $f_{\lambda}(x)$.
\end{rem}
\par In the following proposition, it is shown that a period-doubling bifurcation occurs at $\lambda =\lambda^*$ for $f_{\lambda}(x)$. 
\begin{prop} \label{t6}
Let $\displaystyle f_{\lambda}(x) = \frac{\sin{x}}{x^2 + \lambda}$ for $\lambda >0$. The period-doubling bifurcation occurs at $ \lambda^*$ for $f_{\lambda^*}(x)$.
\end{prop}
\begin{proof} 
Consider $j_{\lambda}(x)= f_{\lambda}^2(x) -x$ for $x \in [0,\pi]$. Now, by Proposition \ref{t2},
\begin{equation*} 
j_{\lambda}(x)
  	\begin{cases}
  		>0 &  \text{for} \; x \in (0,x_{\lambda}),   \\
  		=0 &  \text{at} \; x=x_{\lambda}, \\
  		<0 &   \text{for} \; x \in (x_{\lambda}, \pi),
  	\end{cases}
\end{equation*}
 for $\lambda \in (\lambda^*, 1)$.
We have $j'_{\lambda}(x)= f'_{\lambda}(f_{\lambda}(x))f'_{\lambda}(x) -1$. Clearly, $j'_{\lambda}(x_{\lambda})>0$ when $\lambda \in (0,\lambda^*)$. Thus there exists $\delta>0$ such that 
\begin{equation*} 
j_{\lambda}(x)
  	\begin{cases}
  		<0 &  \text{for} \; x \in (x_{\lambda}- \delta,x_{\lambda}),   \\
  		=0 &  \text{at} \; x=x_{\lambda}, \\
  		>0 &   \text{for} \; x \in (x_{\lambda}, x_{\lambda}+ \delta),
  \end{cases}
\end{equation*}
for $\lambda \in (0,\lambda^*)$. Note that $j_{\lambda}(0)=0$, $j'_{\lambda}(0)>0$ and $j_{\lambda}(\pi)=-\pi$. Therefore, there exists $\epsilon>0$ such that $f_{\lambda}$ has a 2-periodic point in the small neighbourhood of $x_{\lambda}$ for $\lambda \in (\lambda^*- \epsilon, \lambda^*)$ (see in Figure \ref{fof}). Meanwhile, $x_{\lambda}$ is a repulsive fixed point. Hence, the result holds.
\end{proof}
\subsection{Dynamics of $f_{\lambda}(iy)$ and bifurcations}\label{2.2}
The dynamics of the function $f_{\lambda}(z)$ on the imaginary line are the same as the dynamics of $ h_{\lambda}(y) = \frac{\sinh{y}}{ \lambda -y^2}$ for $y \in \mathbb{R}\setminus \{-\sqrt{\lambda}, \sqrt{\lambda}\}$ since $f_{\lambda}(iy)= -ih_{\lambda}(iy)$.
The set $\{-\sqrt{\lambda}, \sqrt{\lambda}\}$ is denoted by $P$.
\subsubsection{Dynamics of $h_{\lambda}(y)$ and bifurcations}
The following subsection explores the dynamics of $h_{\lambda}(y)$.\\
 It is clear that the function $h_{\lambda}(y)$ is odd and continuously differentiable in $\mathbb{R} \setminus P$. Also, notice that $h_{\lambda}(y) \rightarrow -\infty$ as $y \rightarrow \infty$, and  $h_{\lambda}(y) \rightarrow -\infty$ as $y>0$ and $y \to \sqrt{\lambda}$ and $h_{\lambda}(y) \rightarrow \infty$ as $y<0$ and $y \to \sqrt{\lambda}$.
 By (\ref{g2}), the function $h_{\lambda}(y)$ has a local maximum value $ -\frac{\cosh{c_{\lambda}}}{2c_{\lambda}}$ in  $(\sqrt{\lambda}, \infty)$ and local minimum value $ \frac{\cosh{c_{\lambda}}}{2c_{\lambda}}$ in  $(-\infty, -\sqrt{\lambda})$,  hence,  $h_{\lambda}(y)<0$ for $y \in (\sqrt{\lambda}, \infty)$ and $h_{\lambda}(y)>0$ for $y \in (-\infty, -\sqrt{\lambda})$. Now, we have $$ h_{\lambda}''(y)= \frac{(\lambda -y^2)(\lambda-y^2 +2)\sinh{y}  +4y((\lambda -y^2)\cosh{y} + 2y \sinh{y})}{(\lambda -y^2)^3}$$
for $y \in \mathbb{R} \setminus P$. 
  Note that $h_{\lambda}''(y)$ is odd. Let $N(y)$  represent the numerator of the function $h_{\lambda}''(y)$. Observe that $N(y)>0$ for $y \in (0,\sqrt{\lambda})$. Also, we have $N'(y) = (x^4 -(6+2\lambda) x^2 + \lambda^2 +6\lambda  )\cosh x + 12x\sinh x$ and $N''(y) =12y \tanh^2{y} +12\tanh{y} -4y(\lambda -y^2) > 0$ for $y \in (\sqrt{\lambda}, \infty)$. Note that $N'(\sqrt{\lambda})>0$ and $N(\sqrt{\lambda})>0$.
Therefore, 
\begin{equation}
	h''_{\lambda}(y) 
	\begin{cases}
		>0 &  \text{for} \; y \in (-\infty, -\sqrt{\lambda}),   \\
		<0 &  \text{for} \; y \in ( -\sqrt{\lambda},0),   \\
		=0 &  \text{for} \; y=0,\\
		>0 &  \text{for} \; y \in ( 0, \sqrt{\lambda}),   \\
		<0 &   \text{for} \; y \in (\sqrt{\lambda}, \infty).
	\end{cases}													\label{ie2}
\end{equation}
\par	Following the analogous arguments of proof of Proposition \ref{t1}, the following proposition describes the multiplier of the fixed points of $h_{\lambda}(y)$.
\begin{prop} \label{t8}
	Let $ \displaystyle h_{\lambda}(y) = \frac {\sinh{y}}{\lambda- y^2}$ for $y \in \mathbb{R}\setminus P$ and $\lambda >0$. 
	\begin{enumerate}[label=(\alph*)]
		\item For $ \lambda < 1$, $0$ is the repulsive fixed point of $h_{\lambda}(y)$.
		\item For $ \lambda = 1$,  $0$ is the rationally indifferent fixed point of $h_{\lambda}(y)$.
		\item For $\lambda>1$, $h_{\lambda}(y)$ has two non-zero repulsive fixed points $-r_{\lambda}$ and $r_{\lambda}$, and $0$ is an attractive fixed point.
	\end{enumerate}
\end{prop}
\begin{rem}
	Using Proposition \ref{t8}, it is observed that $h_{\lambda}(y)$ has a pitchfork subcritical bifurcation at $\lambda=1$.
\end{rem}
\par Following the analogous argument of proof of Proposition \ref{t1}, the following lemma describes the multiplier of fixed points of $-h_{\lambda}(y)$.
\begin{lem} \label{l1}
	Let $ \displaystyle h_{\lambda}(y)=  \frac{\sinh{y}}{\lambda- y^2}$ for $y \in \mathbb{R}\setminus P$ and $\lambda >0$. 
	There exist $1< \lambda_{1}< \lambda_{2}$ such that:
	\begin{enumerate}[label=(\alph*)]
		\item For $ \lambda= \lambda_{1}$, $-h_{\lambda}(y)$ has two non-zero rationally indifferent $-y_{\lambda}$ and $y_{\lambda}$ and two non-zero repulsive $-r_{\lambda,2}$ and $r_{\lambda,2}$ fixed points. See in Figure \ref{fi}~(a).
		\item For $\lambda_{1} < \lambda <\lambda_{2}$, $-h_{\lambda}(y)$ has two non-zero attractive $-a_{\lambda,2}$ and $a_{\lambda,2}$ and two non-zero repulsive $-r_{\lambda,2}$ and $r_{\lambda,2}$ fixed points. See in Figure \ref{fi}~(b).
		\item For $\lambda = \lambda_{2}$, $-h_{\lambda}(y)$ has two non-zero rationally indifferent $-y_{\lambda}$ and $y_{\lambda}$ fixed points. See in Figure \ref{fi}~(c).
		\item  For $\lambda > \lambda_{2}$, $-h_{\lambda}(y)$ has no non-zero fixed points. See in Figure \ref{fi}~(d).
	\end{enumerate}
\end{lem}
\begin{figure}[h!]
	\centering
    \subfloat[]{\includegraphics[width=0.45\textwidth]{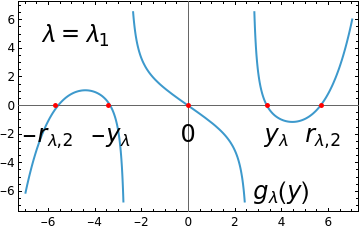}}
    \hspace{1cm}
	\subfloat[]{\includegraphics[width=0.45\textwidth]{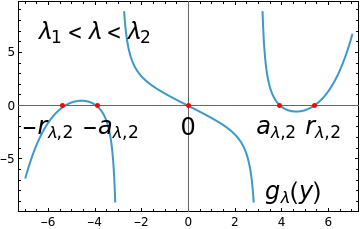}} \\ 
	\subfloat[]{\includegraphics[width=0.45\textwidth]{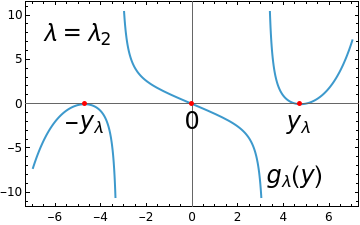}} 
    \hspace{1cm}
	\subfloat[]{\includegraphics[width=0.45\textwidth]{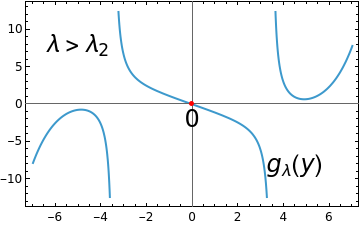}}
	\caption{The graphs of $g_{\lambda}(y)=-h_{\lambda}(y)-y$ for \text{(a)}\; $\lambda= \lambda_1,$  \text{(b)}\; $\lambda_1< \lambda < \lambda_2,$ \text{(c)} \; $\lambda= \lambda_2$ and $\text{(d)} \;\lambda > \lambda_2$.}
	\label{fi}
\end{figure}
\begin{rem} \label{rk1}
Using Proposition \ref{t8} and Lemma \ref{l1}, 	$\{-r_{\lambda,2},\; -a_{\lambda,2}, $ $\;  a_{\lambda,2}, \; r_{\lambda,2} \}$ are the non-zero fixed points of $-h_{\lambda}(y)$ but not of $h_{\lambda}(y)$ for $\lambda_{1} <\lambda < \lambda_{2}$. Further, 
$$  -r_{\lambda,2}< -a_{\lambda,2} < -\sqrt{\lambda}< -r_{\lambda} < 0 < r_{\lambda} < \sqrt{\lambda} <  a_{\lambda,2}< r_{\lambda,2}. $$
\end{rem}
\par The following lemma shows the non-existence of 2-periodic points of $-h_{\lambda}(y)$ for $\lambda> \lambda_2$.
\begin{lem}\label{l2}
  Let $h_{\lambda}(y)=  \frac{\sinh{y}}{\lambda- y^2}$ for $\lambda >\lambda_2$. The function $-h_{\lambda}(y)$ has no $2$-periodic points. 
\end{lem} 
\begin{proof} 
Consider $g_{\lambda}(y)= h_{\lambda}(y)-y$. Since $g'_{\lambda}(0)<0$, $g''_{\lambda}(y)>0$ for $y \in (0,\sqrt{\lambda})$ and $g'_{\lambda}(y) \to \infty$ as $y \to \sqrt{\lambda}^{-}$, there exists a unique critical point of $g_{\lambda}(y)$ in the interval $(0,\sqrt{\lambda})$. Therefore, by Proposition \ref{t8}, 
\begin{equation*}
	g_{\lambda}(y) 
	\begin{cases}
		<0 &  \text{for} \; y \in (0,r_{\lambda}),   \\
		=0 &  \text{at} \; y = r_{\lambda},   \\
		>0 &  \text{for} \; y \in ( r_{\lambda}, \sqrt{\lambda}).
	\end{cases}											\label{ie2}
\end{equation*}
Further, 
\begin{equation*}
	h_{\lambda}(y) 
	\begin{cases}
		<y &  \text{for} \; y \in (0,r_{\lambda}),   \\
		=y &  \text{at} \; y = r_{\lambda},   \\
		>y &  \text{for} \; y \in ( r_{\lambda}, \sqrt{\lambda}).
	\end{cases}											\label{ie2}
\end{equation*}
If $y \in (0,r_{\lambda})$, then the sequence $\{h^n_{\lambda}(y)\}_{n>0}$ is strictly decreasing and bounded below by $0$. Therefore, by Monotonic convergence theorem, $\{h^n_{\lambda}(y)\}_{n>0}$ converges to $0$. Similarly, $\{h^n_{\lambda}(y)\}_{n>0}$ converges to $0$ for  $y \in (-r_{\lambda},0)$.  Note that $|h^{n}_{\lambda}(y)|= |(-h_{\lambda})^{n}(y)|$ for every $n \in \mathbb{N}$. 
Thus,  $\{(-h_{\lambda})^n(y)\}_{n>0}$ converges to $0$ for $y \in (-r_{\lambda},r_{\lambda})$. Further, for $y \in (r_{\lambda}, \sqrt{\lambda})$,  $h_{\lambda}(y)>y$ and hence, there exists $k>0$ such that $h^{k}_{\lambda}(y)> \sqrt{\lambda}$. Since $|h_{\lambda}(c_{\lambda})| > \sqrt{\lambda}$, $|(-h_{\lambda})^{k+1}(y)|> \sqrt{\lambda}$ for $y \in (r_{\lambda}, \sqrt{\lambda})$.\\
We claim that if $y \in (-\infty, - \sqrt{\lambda})$,  then $(-h_{\lambda})^n(y) \rightarrow -\infty$, and if $y \in ( \sqrt{\lambda}, \infty)$, then $(-h_{\lambda})^n(y) \rightarrow \infty$ as $n \rightarrow \infty$.\\
Note that $-h_{\lambda}(y) \to \infty$ as $y \to \infty$ and $-h'_{\lambda}(y) \to \infty$ as $y \to \sqrt{\lambda}^{-}$ (or $\infty$). By Part (d) of Lemma \ref{l1}, $-h_{\lambda}(y)$ has no fixed points for $y \in (\sqrt{\lambda}, \infty)$, so,  $-h_{\lambda}(y)>y$. Also, by (\ref{g2}), $\{ (-h_{\lambda})^n(y)\}_{n>0}$ is strictly increasing and unbounded above. Therefore, $ (-h_{\lambda})^n(y) \to \infty$ as $n \to \infty$. Similarly, if $y \in (-\infty, -\sqrt{\lambda})$, then  $ (-h_{\lambda})^n(y) \to -\infty$ as $n \to \infty$.\\
Hence, the result holds.
\end{proof}
\par In the following proposition,  the multiplier of non-zero 2-periodic points of $h_{\lambda}(y)$ is established by Lemma \ref{l1} and Lemma \ref{l2}. 
\begin{prop} \label{t9}
Let $ \displaystyle h_{\lambda}(y)=   \frac{\sinh{y}}{\lambda- y^2}$ for $\lambda >0$.
\begin{enumerate}[label=(\alph*)]
\item If $\lambda= \lambda_1$, then  $h_{\lambda}(y)$ has a rationally indifferent $y_{\lambda}$ and a repulsive $r_{\lambda,2}$ $2$-periodic points.
\item If  $\lambda_1< \lambda< \lambda_{2}$, then $a_{\lambda,2}$ is an attractive and  $r_{\lambda,2}$ is a  repulsive $2$-periodic points of $h_{\lambda}(y)$.
\item If $\lambda= \lambda_{2}$, then $y_{\lambda}$ is a rationally indifferent $2$-periodic point of $h_{\lambda}(y)$. 
\item If $\lambda> \lambda_{2}$, then  $h_{\lambda}(y)$ has no $2$-periodic points.
\end{enumerate}
\end{prop}
\begin{proof}
Observe that $h^2_{\lambda}(y) = (-h_{\lambda})^2(y)$. So, by Lemma \ref{l1}, Lemma \ref{l2} and Remark \ref{rk1}, the set of non-zero fixed points of $-h_{\lambda}(y)$ are the set of $2$-periodic points of $h_{\lambda}(y)$.   Therefore, (a), (b), (c), and (d) hold.	
\end{proof}
\begin{rem}
Using Proposition \ref{t9}, a tangent bifurcation occurs at $\lambda = \lambda_{2}$ for $h_{\lambda}(y)$.
\end{rem}
\par The following proposition describes the dynamics of $h_{\lambda}(y)$ for $\lambda \geq \lambda_{1}$.
\begin{prop} \label{t13}
Let $ \displaystyle h_{\lambda}(y)=   \frac{\sinh{y}}{\lambda- y^2}$ for $y \in \mathbb{R}\setminus P$ and $\lambda >0$. If $\lambda >1$, then $h^n_{\lambda}(y) \to 0$ for $y \in (-r_{\lambda}, r_{\lambda})$.  
\begin{enumerate}[label=(\alph*)]
\item  If $\lambda =\lambda_1$, then $h_{\lambda}^{n}(y) \rightarrow y_{\lambda} $ as $n \rightarrow \infty$  for $y \in (r'_{\lambda,2},r_{\lambda,2})$, where $r'_{\lambda,2}$ is the preimage of $r_{\lambda,2}$ under $-h_{\lambda}(y)$. Also, if $y \in (\sqrt{\lambda},r'_{\lambda,2}) \cup (r_{\lambda,2}, \infty) $, then $|h_{\lambda}^{n}(y)| \rightarrow \infty$ as $n \rightarrow \infty$.
\item If $ \lambda_{1} <\lambda < \lambda_{2}$, then $h_{\lambda}^{n}(y) \rightarrow a_{\lambda,2} $ as $n \rightarrow \infty$  for $y \in (r'_{\lambda,2},r_{\lambda,2})$, where $r'_{\lambda,2}$ is the preimage of $r_{\lambda,2}$ under $-h_{\lambda}(y)$. Also, if $y \in (\sqrt{\lambda},r'_{\lambda,2}) \cup (r_{\lambda,2}, \infty) $, then $|h_{\lambda}^{n}(y)| \rightarrow \infty$ as $n \rightarrow \infty$.
\item  If $\lambda =\lambda_2$, then $h_{\lambda}^{n}(y) \rightarrow y_{\lambda} $ as $n \rightarrow \infty$  for $y \in (y'_{\lambda},y_{\lambda})$, where $y'_{\lambda}$ is the preimage of $y_{\lambda}$ under $-h_{\lambda}(y)$. Also, if $y \in (\sqrt{\lambda},y'_{\lambda}) \cup (y_{\lambda}, \infty) $, then $|h_{\lambda}^{n}(y)| \rightarrow \infty$ as $n \rightarrow \infty$.
\item If $\lambda > \lambda_{2}$, then $\left|h^n_{\lambda}(y)\right| \to \infty$ as $n \rightarrow \infty$ for $y \in (-\infty, -r_{\lambda}) \setminus M \cup (r_{\lambda}, \infty) \setminus N$. Here, $M$ and $N$ are the sets of all preimages of $-\sqrt{\lambda}$ and $\sqrt{\lambda}$ under $h_{\lambda}(y)$, respectively.		
\end{enumerate}
\end{prop}
\begin{proof} 
By Lemma \ref{l2}, for $\lambda >1$, $h^n_{\lambda}(y) \to 0$ for $y \in (-r_{\lambda}, r_{\lambda})$. \\
 (b) For $\lambda_{1}< \lambda< \lambda_{2}$:\\
	 Let $r'_{\lambda,2}$ be the preimage of $r_{\lambda,2}$ under $-h_{\lambda}(y)$. Notice that $-h_{\lambda}(y)$ is unimodal on $[r'_{\lambda,2},r_{\lambda,2}]$. By parallel arguments of proof of Proposition \ref{t2}, 
	   $(-h_{\lambda})^{n}(y) \rightarrow a_{\lambda,2}$ as $n \rightarrow \infty$ for $y \in (r'_{\lambda,2},r_{\lambda,2})$. Therefore,  $(h_{\lambda})^{2n}(y) \rightarrow a_{\lambda,2}$ as $n \rightarrow \infty$ for $y \in (r'_{\lambda,2},r_{\lambda,2})$, since $h^2_{\lambda}(y) = (-h_{\lambda})^2(y)$. By Lemma \ref{l1}, $-h_{\lambda}(y)$ has no fixed points for $y \in (r_{\lambda,2}, \infty)$. Also, $-h_{\lambda}(y)$ is strictly increasing on $(r_{\lambda,2}, \infty)$, $(-h_{\lambda})'(r_{\lambda,2})>1$ and $-h_{\lambda}(y) \to \infty$ as $y \to \infty$. Therefore, $-h_{\lambda}(y)>y$. Further, if $y \in (\sqrt{\lambda},r'_{\lambda,2}) \cup (r_{\lambda,2}, \infty)$, then $(-h_{\lambda})^{n}(y) \rightarrow \infty$ as $n \rightarrow \infty$. Since $h^{2n}_{\lambda}(y)= (-h_{\lambda})^{2n}(y)$ and $-h^{2n-1}_{\lambda}(y)= (-h_{\lambda})^{2n-1}(y)$ for every $n \in \mathbb{N}$, $h^{2n}_{\lambda}(y) \to \infty$ and $h^{2n-1}_{\lambda}(y) \to -\infty$ as $n \to \infty$ when $y \in (\sqrt{\lambda},r'_{\lambda,2}) \cup (r_{\lambda,2}, \infty)$. Hence, the result holds.\\
 (a)  Using parallel arguments of Part (b) of Proposition \ref{t13}, the result holds. \\
 (c)  Using parallel arguments of Part (b) of Proposition \ref{t13}, the result holds.\\
 (d)  Using parallel arguments of Lemma \ref{l2}, the result holds.
 \end{proof}
\section{Results} \label{2.3}
For the $p$-periodic attractive cycle   $\{z_0,\cdots, z_{p-1}\}$ of $z_0$, the basin of attraction
$$A_p(z_0) := \bigcup \limits_{j=0}^{p-1}   A(z_j,f^p), $$
where $A(z_j,f^p) := \{ z \in \mathbb{C} : f^{np}(z) \to z_j \; \text{as} \; n \to \infty \}$ and $0 \leq j \leq p-1$.
\par The following theorem describes the Fatou set $\mathcal{F}(f_{\lambda})$  for $\lambda>\lambda_1$.
\begin{thm}\label{t17}
Let $f_{\lambda} \in \mathbb{S}$.
\begin{enumerate}[label=(\alph*)]
\item If $\lambda_{1}< \lambda< \lambda_{2}$, then the Fatou set  $\mathcal{F}(f_{\lambda}) = A_1(0) \cup A_2(ia_{\lambda,2})$, where ${a_{\lambda,2}}$ is defined in Proposition \ref{t9}. 
\item If $\lambda> \lambda_{2}$, then the Fatou set $\mathcal{F}(f_{\lambda}) = A_1(0)$. 
\end{enumerate}
\end{thm}	
\begin{proof}
Using Proposition \ref{p2}, the set of singularities of $f^{-1}_{\lambda}$ is contained either on the real line or the imaginary axis and is bounded, that means, $f_{\lambda} \in \mathcal{B}$. Also, the function $f_{\lambda}$ has only two critical values on the imaginary axis. Using Proposition \ref{t2}, the real line is contained in the basin of attraction of $0$,  $A_1(0)$ for $\lambda>1$. Using Proposition \ref{t9}, $f_{\lambda}$ has an attractive $2$-periodic point $ia_{\lambda,2}$ and therefore, the Fatou set $\mathcal{F}(f_{\lambda})$ contains the basin of attraction of $2$-periodic point $ia_{\lambda,2}$,  $A_2(ia_{\lambda,2})$. Also, by Proposition \ref{t13}, all the imaginary critical values of $f_{\lambda}$ are contained in $A_2(ia_{\lambda,2})$. 
Therefore,
$$\mathcal{P}(f_{\lambda}) \subset A_1(0) \cup A_2(ia_{\lambda,2}).$$ Consequently, $ (\mathcal{P}(f_{\lambda}))' \cap J^{\lambda}_{\infty} \setminus \{\infty\}= \emptyset$ for $\lambda \in (\lambda_1, \lambda_2)$. Further, by Part (d) of Proposition \ref{t13}, $\mathcal{J}(f_{\lambda})\cap (\mathcal{P}(f_{\lambda}))'= \{\infty\}$ and $(\mathcal{P}(f_{\lambda}))' \cap J^{\lambda}_{\infty} \setminus \{\infty\}= \emptyset$ for $\lambda> \lambda_2$, where ${J}^{\lambda}_{\infty}= \bigcup \limits_{n=0}^{\infty} f^{-n}_{\lambda}(\infty)$ and $(\mathcal{P}(f_{\lambda}))'$ is the set of accumulation points of  $\mathcal{P}(f_{\lambda})$. Therefore, by Theorem 3 and Theorem 4 of \cite{zheng2003}, $f_{\lambda}$ has no Baker domains and wandering domains  for $\lambda \in ( \lambda_{1},\lambda_2) \cup (\lambda_2, \infty)$.   
\par By Theorem 7 of \cite{Ber93a}, the Fatou set $\mathcal{F}(f_{\lambda})$ does not contain Siegel discs and Herman rings for $\lambda >1$, because $\mathbb{R} \subset A_1(0)$.\\ 
(a) Note that $${S}(f_{\lambda}) \subset A_1(0) \cup A_2(ia_{\lambda,2})$$ for $\lambda \in (\lambda_1,\lambda_2)$. By Theorem 7 of \cite{Ber93a}, the Fatou set $\mathcal{F}(f_{\lambda})$ does not contain parabolic domains. Similarly, $\mathcal{F}(f_{\lambda})$ does not contain any basin of attraction other than $A_1(0)$ and $A_2(ia_{\lambda,2})$. Hence, the result holds.\\
\begin{figure}[h!]
\centering
\includegraphics[width=12cm, height=8cm]{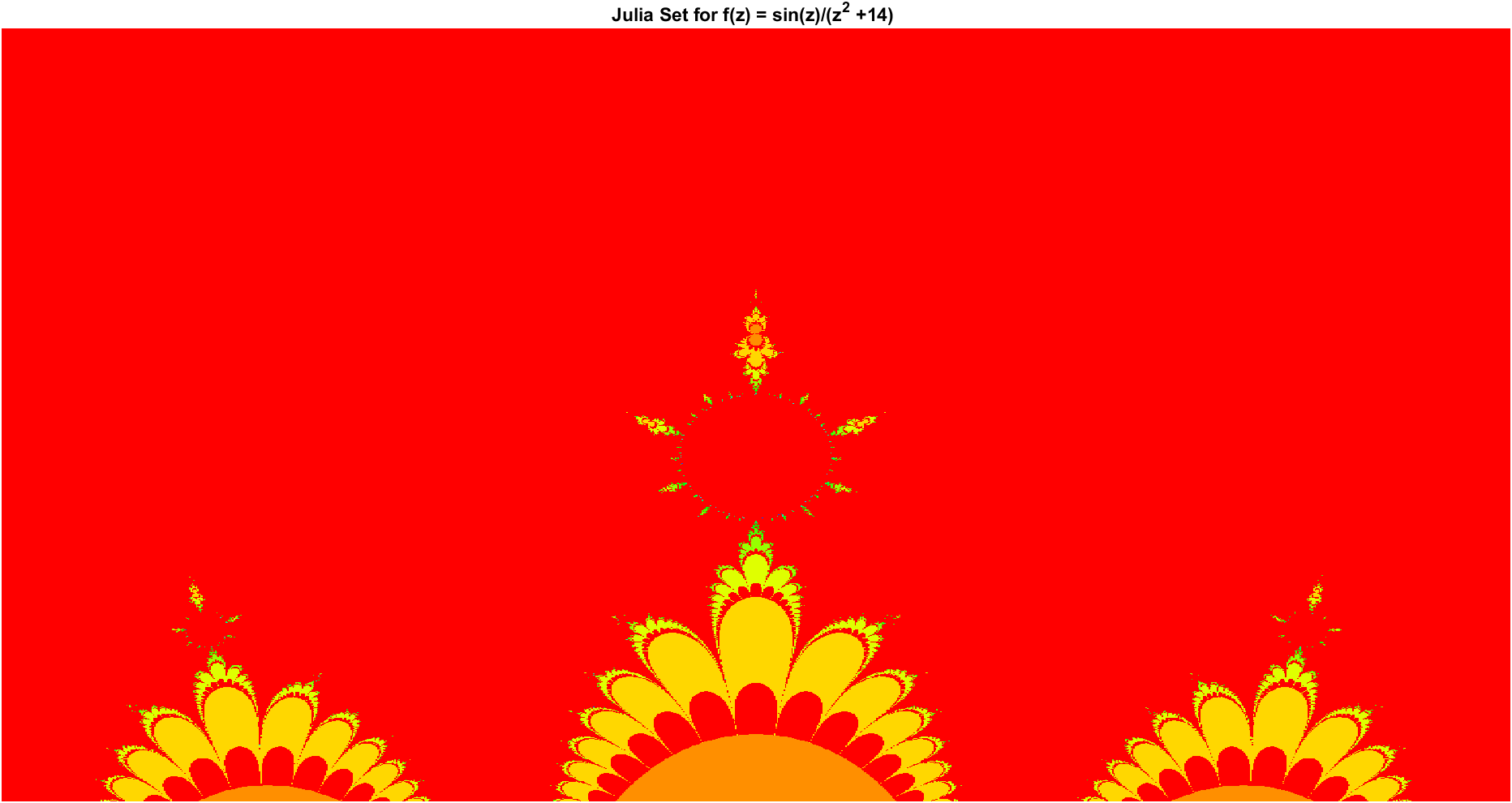}
\caption{{Red colour region} represents the Fatou set of $f_{\lambda}$ for $\lambda_{1}< \lambda< \lambda_{2}$. Here $-1.5 \pi <x< 1.5 \pi$ and $-2 \pi< y <0$.}
\label{J1}
\end{figure}
\begin{figure}[h!]
\centering
\includegraphics[width=12cm, height=8cm]{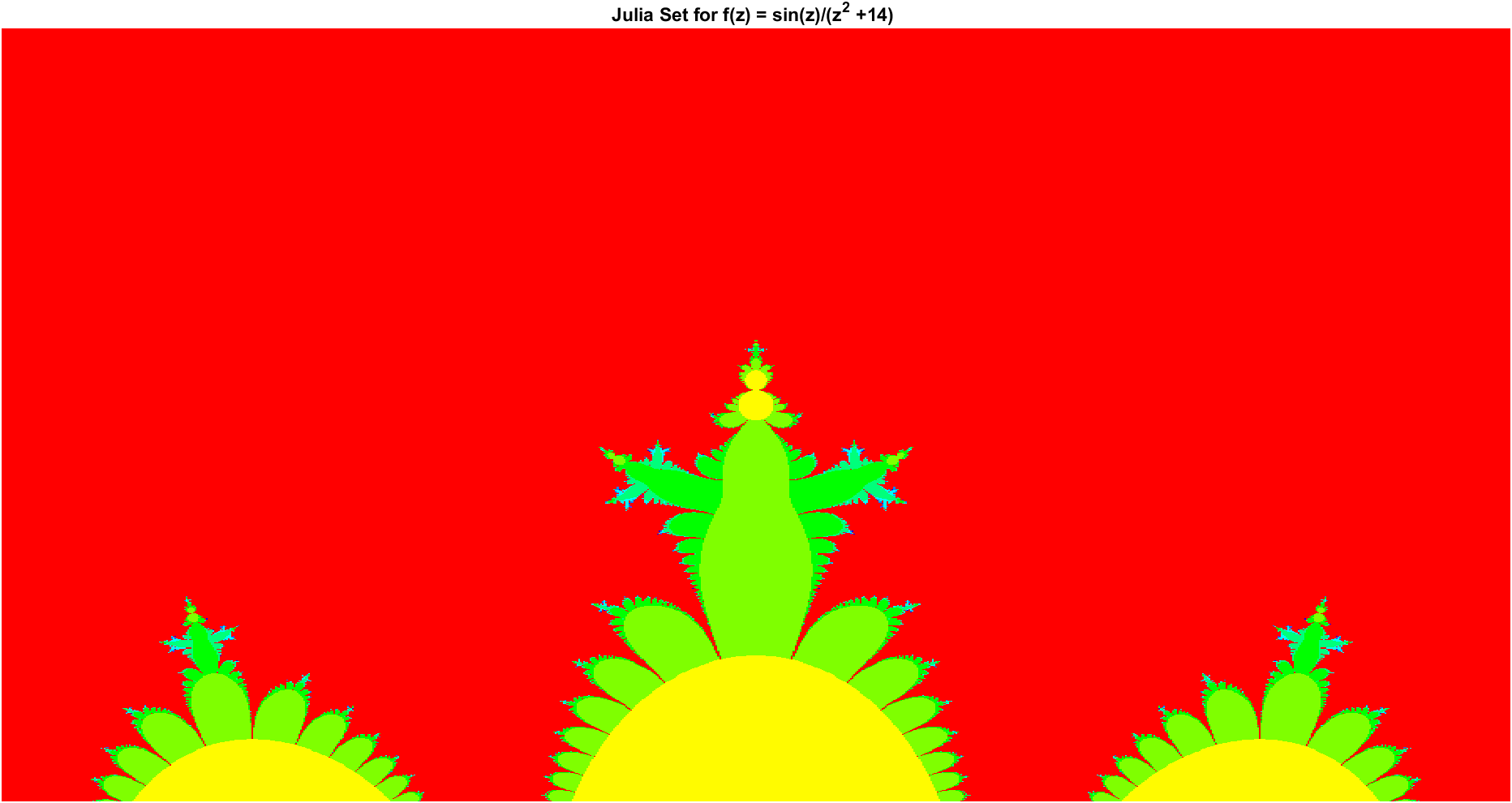}
\caption{ Red color region represents the Fatou set of  $f_{\lambda}$ for $\lambda> \lambda_{2}$. Here $-1.5 \pi <x< 1.5 \pi$ and $-2 \pi< y <0$.}
\label{J2}
\end{figure} 
(b) By Part (b) of Proposition \ref{t13}, the orbit of imaginary critical values of $f_{\lambda}$ tends to $\infty$ for $\lambda> \lambda_2$. Therefore, by Theorem 7 of \cite{Ber93a}, the Fatou set $\mathcal{F}(f_{\lambda})$ does not contain parabolic domains. Similarly, $\mathcal{F}(f_{\lambda})$ does not contain any basin of attraction other than $A_1(0)$. Hence, the result holds. 
\end{proof} 
By Proposition \ref{t1}, the Fatou set $\mathcal{F}(f_{\lambda})$ is symmetric with respect to the real and imaginary axes. Figures \ref{J1} and \ref{J2} illustrate the basins of attraction of $f_{\lambda}$ when $\lambda \in (\lambda_{1},\infty) \setminus \{\lambda_{2}\}$. 
\begin{rem}\label{r3}
\begin{enumerate}
 \item By Part (a) of Theorem \ref{t17}, the function $f_{\lambda}$ is topologically hyperbolic for ${\lambda_{1}} <\lambda < \lambda_{2}$. 
\item For $\lambda \in \{\lambda_1, \lambda_2\}$:\\
Note that the set $\mathcal{P}(f_{\lambda})$ is bounded since all the real singularities of $f^{-1}_{\lambda}$ are contained in the basin of attraction of $0$ of $f_{\lambda}$ and the remaining imaginary singularities of $f^{-1}_{\lambda}$ are contained in parabolic domains corresponding to the rationally indifferent $2$-periodic point $i y_{\lambda}$. Therefore, by Corollary 1 of \cite{zheng2003},  $f_{\lambda}$ has no Baker domains. Further, by the parallel arguments of proof of Theorem \ref{t17}, $\mathcal{F}(f_{\lambda})$ does not contain any other Fatou components except $A_1(0)$, parabolic domains of $i y_{\lambda}$ and their premiages, and wandering domains.
\end{enumerate}
  Using the similar approach of \cite{gpt}, in the following theorem, it is shown that  $\mathcal{J}(f_{\lambda})= clos({I(f_{\lambda})})$ for $\lambda \in (\lambda_1, \infty) \setminus \{\lambda_2\}$, where $I(f_{\lambda}):= \{ z \in \mathbb{C}: f_{\lambda}^{n}(z) \to \infty \;  \text{as} \; n \to \infty \; \text{but} \;f_{\lambda}^{n}(z)\neq \infty \; \text{for any} \; n \}$. 
  \begin{thm}\label{t21}\textbf{(Characterization of Julia set)}
    Let $f_{\lambda} \in \mathbb{S}$. For $\lambda \in (\lambda_1, \infty) \setminus \{\lambda_2\}$, $\mathcal{J}(f_{\lambda})= clos({I(f_{\lambda})})$. 
  \end{thm}
  \begin{proof}
     Suppose $z \in clos({I(f_{\lambda}) })$  but $z \notin \mathcal{J}(f_{\lambda})$, then for a given neighbourhood $U$ of $z$ there exists $z_1 \in U$ such that $f_{\lambda}^n(z_1) \to \infty$ as $n \to \infty$. By Theorem \ref{t17}, $z$ belongs to either the basin of attraction of $0$ or the basin of attraction of $i a_{\lambda, 2}$ (2-periodic point) depends upon $\lambda$. Therefore, it contradicts our supposition. Hence, $clos({I(f_{\lambda})}) \subseteq \mathcal{J}(f_{\lambda})$.\\
     Suppose $z \in \mathcal{J}(f_{\lambda})$, then for a given neighbourhood of $z$, say $U$, $\bigcup \limits_{n=1}^{\infty} f_{\lambda}^{n}(U)$ omit at most two points in $\mathbb{C}$. In particular, there exists a point $i \tilde{x}_{\lambda} 
 (\tilde{x}_{\lambda}> r_{\lambda} \; (\text{or} \; r_{\lambda,2}))$ such that $i \tilde{x}_{\lambda} \in \bigcup \limits_{n=1}^{\infty} f_{\lambda}^{n}(U)$, where $r_{\lambda} \; \text{and} \; r_{\lambda,2}$ (depends on $\lambda$) are the repulsive fixed point of $f_{\lambda}$. Therefore, there exists a point $\tilde{z} \in U$ such that $f_{\lambda}^j(\tilde{z})= i\tilde{x_{\lambda}}$ for some positive integer $j$. Now, by Proposition \ref{t13}, it follows that $f_{\lambda}^n (i\tilde{x}_{\lambda}) \to \infty$ as $n \to \infty$, since $\tilde{x}_{\lambda}> r_{\lambda}$ (or $r_{\lambda,2}$). Thus $f_{\lambda}^n(\tilde{z}) \to \infty$ as $n \to \infty$ and hence, $z \in clos({I(f_{\lambda})})$. Consequently, $\mathcal{J}(f_{\lambda}) \subseteq clos({I(f_{\lambda})})$.     
  \end{proof}
\end{rem}
Using Proposition \ref{t1} and following the analogous arguments of proof of Theorem \ref{t17},  the following theorem describes the Fatou components of $f_{\lambda}$ for $\lambda> \lambda^*$. 
\begin{thm}\label{t20}
Let $f_{\lambda} \in \mathbb{S}$.
    \begin{enumerate}
    \item For $\lambda =\lambda^*$, the Fatou set $\mathcal{F}(f_{\lambda})$ contains the union of parabolic domains of rationally indifferent fixed points $-x^*$ and $x^*$ and their preimages, where $x^*$ is defined in Proposition \ref{t1}.
    \item For $\lambda^*< \lambda <1 $, $A_1(-x_\lambda)\cup A_1(x_\lambda)  \subseteq \mathcal{F}(f_{\lambda})$, where $x_{\lambda}$ is defined in Proposition \ref{t1}.
    \item For $\lambda =1$, $\mathcal{F}(f_{\lambda})$ contains the union of parabolic domains of rationally indifferent fixed point $0$ of $f_{\lambda}$ and their preimages.
    \item For $\lambda >1$, $A_1(0) \subseteq \mathcal{F}(f_{\lambda})$.
    \item For $\lambda \geq \lambda^*$, $\mathcal{F}(f_{\lambda})$ does not contain Siegel disks and Herman rings.
    \end{enumerate}
\end{thm}
\begin{rem}
    Using Proposition  \ref{t7} and Theorem 4.1 of \cite{Bernd}, $f_{\lambda}(x)$ shows sensitive dependence on $[-f_{\lambda}(p_{\lambda}), f_{\lambda}(p_{\lambda})]$ for $\lambda \in (0, \lambda^{**}]$. Therefore, $\mathbb{R} \subset \mathcal{J}(f_{\lambda})$  for $\lambda \in (0, \lambda^{**}]$ since $f_{\lambda}(\mathbb{R}) \subseteq [-f_{\lambda}(p_{\lambda}), f_{\lambda}(p_{\lambda})]$.
\end{rem}
The following theorems describe some topological aspects of the Fatou components of $f_{\lambda}$ for a certain range of $\lambda$ values. Given a non-empty set $A$, $-A:= \{ -z : z \in A\}$ and $\overline{A}:= \{ \bar{z} : z \in A\}$. 
\begin{thm}\label{t18}
 Let $f_{\lambda} \in \mathbb{S}$. For $\lambda> 1$, $\mathcal{F}(f_{\lambda})$ has a completely invariant attracting domain, $A_1(0)$. Further, the degree of $f_{\lambda}$ on $ A_1(0)$, $deg(f_{\lambda| A_1(0)})= \infty$.  
\end{thm}
\begin{proof}
It is easy to see that the set of preimages of zero $f^{-1}_{\lambda}(0)= \{n\pi : n \in \mathbb{Z}\}$ is infinite and so, $deg(f_{\lambda| A_1(0)})= \infty$. Using  Proposition \ref{t2},  $ \mathbb{R} \subset A_1(0)$ for $\lambda> 1$. Therefore, there exists a component $U$ of $A_{1}(0)$ such that $\mathbb{R} \subset U \subseteq A_{1}(0)$. Now, consider $\mathbb{D}(0,r) \subset U$ for some $r>0$ and $U_r$ as a component of $f^{-1}_{\lambda}(\mathbb{D}(0,r))$. If $U_r \cap \mathbb{R} \neq \emptyset$, then $U_r \subseteq U$. 
\par Suppose $U_r \cap \mathbb{R} = \emptyset$. Using Proposition \ref{p2}, $U_r$ contains an asymptotic curve and thus, there is a direct singularity of $f_{\lambda}^{-1}$ over $0$. Since $f_{\lambda}(-z) = f_{\lambda}(z)$ and $f_{\lambda}(\bar{z}) = \overline{f_{\lambda}(z)}$, $\overline{U_r}$ and $-U_r$ are the components of $f^{-1}_{\lambda}(\mathbb{D}(0,r))$ such that $\overline{U_r} \cap \mathbb{R} = \emptyset$ and $-{U_r} \cap \mathbb{R} = \emptyset$. Thus, there are at least three direct singularities of $f^{-1}_{\lambda}$ over $0$.  Since the order of $f_{\lambda}$ is one, by Denjoy-Carleman-Ahlfors Theorem of \cite{nev} there are at most two direct singularities of $f^{-1}_{\lambda}$. Therefore, it gives a contradiction. Hence, $U_r \cap \mathbb{R} \neq \emptyset$ for every component of $f^{-1}_{\lambda}(\mathbb{D}(0,r))$. Consequently, $f^{-1}_{\lambda}(U) \subseteq U$ and the result holds.
\end{proof}
\begin{rem}\label{rem9}
1.    By Lemma 2.2 of \cite{Cao}, $\partial A_1(0)= \mathcal{J}(f_{\lambda})$ for $\lambda>1$.\\
2. The inverse of $f_{\lambda}$ has exactly two indirect singularities over $0$ but no direct singularities.
\end{rem}
Consider $\Phi(y)= \frac{\sinh{y}}{y} +y^2$ for $y \in (0,\infty)$. Now, we have $\Phi'(y)= \frac{2y^3 -\sinh y+y\cosh{y}}{y^2}$. Let  $N(y)$  represent the numerator of the function $\Phi'(y)$. Then $N'(y)= y(6y+\sinh (y))>0$ and thus, $N(y)$ is strictly increasing on $(0,\infty)$. Note that $N(0)=0$. Therefore, $\Phi'(y)>0$ for $y \in (0, \infty)$. Hence, $\Phi(y)$ is strictly increasing and $\Phi(y)>1$ for $y \in (0, \infty)$. 
\begin{thm}\label{tt19}
 Let $f_{\lambda} \in \mathbb{S}$. For $\lambda> {1}$, there is a disc
  $$ \mathbb{D}(0,r_{\lambda}) \subset A_1(0),$$
   where $ir_{\lambda}$ is the repulsive fixed point of $f_{\lambda}$. Further, if $\lambda> \sinh({1}) +1$, then all the real singularities of $f^{-1}_{\lambda}$ are contained in the disc $\mathbb{D}(0,r_{\lambda})$.
\end{thm}
\begin{proof}
 Note that $f_{\lambda}$ is analytic on $\mathbb{D}(0,r_{\lambda})$. So, it is sufficient to prove that $\mathbb{D}(0,r_{\lambda})$ is invariant under $f_{\lambda}$, i.e.,  $f_{\lambda}(\mathbb{D}(0,r_{\lambda})) \subseteq \mathbb{D}(0,r_{\lambda})$. By Maximum Modulus Principle, if $f_{\lambda} (\partial \mathbb{D}(0,r_{\lambda})) \subseteq clos({\mathbb{D}(0,r_{\lambda})}) $ holds, then $f_{\lambda}(\mathbb{D}(0,r_{\lambda})) \subseteq \mathbb{D}(0,r_{\lambda})$. So, by substituting $z=r_{\lambda}e^{i \theta}$ and $\lambda= \Phi(r_{\lambda})$ in $f_{\lambda}(z)$, we get
$$
\left|\frac{f_{\lambda}(r_{\lambda} e^{i\theta})}{r_{\lambda}} \right|^2 = \frac{\sin^2(r_{\lambda} \cos \theta) + \sinh^2(r_{\lambda} \sin \theta)}{(r_{\lambda}^3 \cos 2\theta +{\sinh(r_{\lambda})} +r_{\lambda}^3 )^2 + (r_{\lambda}^3 \sin 2\theta)^2},$$
where $\theta \in (0,2 \pi)$. \\
\par \noindent
{Claim:} $\left|\frac{f_{\lambda}(r_{\lambda} e^{i\theta})}{r_{\lambda}} \right|^2 \leq 1$ for $\theta \in (0,2 \pi)$.\\
 Let $N_{\lambda}(\theta)$ represent the numerator of the function $
\left|\frac{f_{\lambda}(r_{\lambda} e^{i\theta})}{r_{\lambda}} \right|^2$. Then we have $N'_{\lambda}(\theta)= r_{\lambda}(-\sin(2r_{\lambda} \cos \theta) \sin{\theta} + \sinh(r_{\lambda} \sin{\theta})\cos{\theta})$ for $\theta \in (0, \pi)$.\\
 \par \noindent
For $\theta \in \left(0, \frac{\pi}{2}\right)$:\\
Since  $-\sin\theta>  -\theta$ and $\cos{\theta}>0$, $-\sin(2r_{\lambda} \cos{\theta}) > -2 r_{\lambda} \cos{\theta}$ and thus, $-(\sin(2r_{\lambda} \cos \theta)) \sin \theta > -2r_{\lambda} \cos{\theta} \sin{\theta}$. Also, $\sinh (2r_{\lambda} \sin \theta) >  2r_{\lambda} \sin \theta$ and thus, $\sinh (2r_{\lambda} \sin \theta) \cos \theta>  2r_{\lambda} \sin \theta \cos \theta$. Therefore, $N'_{\lambda}(\theta)>0$ and it implies that $N_{\lambda}(\theta)$ is strictly increasing on  $\left(0, \frac{\pi}{2}\right)$.\\
 \par \noindent
For $\theta \in \left(\frac{\pi}{2}, \pi\right)$:\\
Note that $-\sin\theta>  -\theta$ and $\cos{\theta}<0$. By parallel arguments for  $\theta \in \left(0, \frac{\pi}{2}\right)$, $N'_{\lambda}(\theta)<0$. Therfore, it implies that $N_{\lambda}(\theta)$ is strictly decreasing on  $\left(\frac{\pi}{2}, \pi\right)$. Hence, $N_{\lambda}(\theta)$ has the maximum value $\sinh^2(r_{\lambda})$ at $\frac{\pi}{2}$.\\
 Let  $D_{\lambda}(\theta)$ represent the denominator of the function  $
\left|\frac{f_{\lambda}(r_{\lambda} e^{i\theta})}{r_{\lambda}} \right|^2$. 
Now, we have $D'_{\lambda}(\theta)= -4 r^3_{\lambda}(\sinh({r_{\lambda}}) +r^3_{\lambda}) \sin(2 \theta)$ for $\theta \in (0, \pi)$. Therefore, $D'_{\lambda}(\theta)<0$ for $\theta \in (0, \frac{\pi}{2})$ and $D'_{\lambda}(\theta)>0$ for $\theta \in \left(\frac{\pi}{2}, \pi\right)$. Hence, $N_{\lambda}(\theta)$ has the minimum value $\sinh^2(r_{\lambda})$ at $\frac{\pi}{2}$. Consequently, our claim holds since $N(\theta)$ and $D(\theta)$ are $\pi$-periodic function.
\par  Further, for $\lambda> \sinh({1}) +1$,  $r_{\lambda}>1$ and also, $|f_{\lambda}(x)|<1$ for $x \in \mathbb{R}$. Therefore, all the real singularities of $f^{-1}_{\lambda}$ are contained in the disc $\mathbb{D}(0,r_{\lambda})$.
\end{proof}
\begin{thm}\label{t19}
 Let $f_{\lambda} \in \mathbb{S}$. For $\lambda \geq \lambda_{1}$, every Fatou component of $f_{\lambda}$ is simply connected.
\end{thm}
\begin{proof}
Suppose the basin of attraction $A_1(0)$ is not simply connected. Then there is a simple closed curve $\gamma$ which lie inside $A_1(0)$ but $(R(\gamma))^{\circ} \cap \mathcal{J}(f_{\lambda}) \neq \emptyset$ for $\lambda \geq \lambda_{1}$, where $R(\gamma)$ is the bounded region contained in the complement of $\gamma$. So, by Lemma 3.1 of \cite{tn}, there is a non-negative integer $p$ such that $R(f_{\lambda}^p(\gamma))$ contains at least one pole of $f_{\lambda}$. Thus either $f_{\lambda}^p(\gamma) \cap (i\sqrt{\lambda}, \infty) \neq \emptyset$ or $f_{\lambda}^p(\gamma) \cap (-\infty,-i\sqrt{\lambda} )\neq \emptyset$, therefore, there exists a $z \in A_1(0)$ such that  either $ z \in (i\sqrt{\lambda}, \infty)$ or $z \in (-\infty,-i\sqrt{\lambda} )$. Without loss of generality, assume $z \in (i\sqrt{\lambda}, \infty)$. Using Proposition \ref{t13}, either $f^n_{\lambda}(z) \to \infty$, or $f^n_{\lambda}(z) \to ia_{\lambda,2}$, or $f^n_{\lambda}(z) \to ir_{\lambda,2}$, or  $f^n_{\lambda}(z) \to iy_{\lambda}$ as $n \to \infty$ for $\lambda \in [\lambda_1, \lambda_{2}]$ and  $f^n_{\lambda}(z) \to \infty$ as $n \to \infty$ for $\lambda \in ( \lambda_{2}, \infty)$. But it is not true, since $z \in A_1(0)$. Therefore, it contradicts our supposition. Hence, $A_1(0)$ is simply connected for $\lambda \geq \lambda_1$.  Consequently, by  Part $1$ of Remark \ref{rem9}, the result holds.
\end{proof}
\begin{center}
\begin{table}[h!]
	\label{table}
	\begin{tabular} {p{3 cm}  p {8cm}  p {4 cm}  }
        \hline
	Parameter range &  Description of Fatou set $\mathcal{F}(f_{\lambda})$ &   Description of Julia set $\mathcal{J}(f_{\lambda})$\\
     \hline
    $0< \lambda\leq \lambda^{**}$  &  $\mathcal{F}(f_{\lambda})$ does not contain  Siegel discs and Herman rings.  &  $\bigcup \limits_{n=0}^{\infty} f^{-n}_{\lambda}(\mathbb{R}) \subset \mathcal{J}(f_{\lambda})$   \\
     \hline
    $\lambda^{**}< \lambda <\lambda^{*}$  &  $\mathcal{F}(f_{\lambda})$ does not contain  Siegel discs and Herman rings.  &  $\bigcup \limits_{n=0}^{\infty} f^{-n}_{\lambda}(\{n\pi | n \in \mathbb{Z}\})  \subset \mathcal{J}(f_{\lambda})$   \\
    \hline
    $\lambda = \lambda^{*}$ &
    
    $\mathcal{F}(f_{\lambda})$ contains the union of parabolic domains of $-x^*$ and $x^*$ and their preimages. Moreover, it does not contain  Siegel discs and Herman rings. &  $\bigcup \limits_{n=0}^{\infty} f^{-n}_{\lambda}(\{n\pi | n \in \mathbb{Z}\} \cup \{-x^*,x^*\})  \subset \mathcal{J}(f_{\lambda})$\\
    \hline 
    $\lambda^*<\lambda<1 $ &  $A_1(-x_\lambda)\cup A_1(x_\lambda)  \subseteq \mathcal{F}(f_{\lambda})$. Moreover, $\mathcal{F}(f_{\lambda})$ does not contain Siegel discs and Herman rings. 
    & $\bigcup \limits_{n=0}^{\infty} f^{-n}_{\lambda}(\{n\pi | n \in \mathbb{Z}\}) \subset \mathcal{J}(f_{\lambda})$ \\
    \hline
     $\lambda =1$ & $\mathcal{F}(f_{\lambda})$ contains the union of parabolic domains of $0$ and their preimages. Moreover, it does not contain  Siegel discs and Herman rings. &  $ \bigcup \limits_{n=0}^{\infty} f^{-n}_{\lambda}(\{n\pi | n \in \mathbb{Z}\}) \subset \mathcal{J}(f_{\lambda})$\\
     \hline 
     $1< \lambda< \lambda_1$ & $A_1(0) \subseteq \mathcal{F}(f_{\lambda})$. Moreover, $\mathcal{F}(f_{\lambda})$ does not contain Siegel discs and Herman rings. & $\partial A_{1}(0) = \mathcal{J}(f_{\lambda})$\\
     \hline
     $ \lambda \in \{ \lambda_1, \lambda_2\}$ & $\mathcal{F}(f_{\lambda})$ contains the union of $A_1(0)$, and parabolic domains of  $2$-periodic points $-iy_{\lambda}$ and $iy_{\lambda}$ and their preimages. Moreover, it does not contain any other Fatou components except the possibility of wandering domains. & $\partial A_{1}(0) = \mathcal{J}(f_{\lambda})$ \\
     \hline 
     $\lambda_1< \lambda< \lambda_2$ &  $\mathcal{F}(f_{\lambda}) = A_1(0) \cup A_2(ia_{\lambda,2})$ & $\partial A_{1}(0) = \mathcal{J}(f_{\lambda})$\\
      \hline 
      $ \lambda> \lambda_2$ &  $\mathcal{F}(f_{\lambda}) = A_1(0)$ & $\partial A_{1}(0) = \mathcal{J}(f_{\lambda})$\\
      \hline
	\end{tabular}
 \caption{Description of the Fatou and Julia sets of $f_{\lambda}$ for various $\lambda>0$.}
  \label{table1}
\end{table}
\end{center}
\begin{rem}
\begin{enumerate}
\item By Theorem \ref{t17},  Theorem \ref{t18}, Remark \ref{r3} and Remark \ref{rem9}, $f_{\lambda}$ has a completely invariant Fatou component $A_1(0)$ which contains all  indirect singularities over $0$ of $f^{-1}_{\lambda}$  but it does not contain imaginary critical values of $f_{\lambda}$ for $\lambda \geq \lambda_1$. 
\item By  Part (a) of Theorem \ref{t17}, and  Theorem  \ref{t18}, the Fatou set $\mathcal{F}(f_{\lambda})$ is strictly contained one completely invariant Fatou component for $\lambda \in (\lambda_{1}, \lambda_{2})$.  By Part (b) of Theorem \ref{t17} and  Theorem \ref{t18}, $\mathcal{F}(f_{\lambda})$ is a  completely invariant domain for $\lambda>\lambda_{2}$. Further, by Theorem \ref{t19}, the Julia set $\mathcal{J}(f_{\lambda})$ is connected in $\widehat{\mathbb{C}}$ for $\lambda \geq \lambda_1$.
\end{enumerate}
 \end{rem}
\section{Conclusion}
In this article, the dynamics of a one-parameter family of functions $f_{\lambda}(z)= \frac{\sin{z}}{z^2+ \lambda},$  $\lambda>0$, are investigated. It is found that the set of singularities of $f^{-1}_{\lambda}$ is bounded. It is shown that the family $f_{\lambda}$ exhibits different types of bifurcation. Also, it is proved that all components of $\mathcal{F}(f_{\lambda})$ are simply connected when $\lambda \geq \lambda_1$. 
The Fatou set $\mathcal{F}(f_{\lambda})$ has a completely invariant single attracting domain containing all indirect singularities of $f^{-1}_{\lambda}$ but not all critical values when $\lambda \geq \lambda_1$. Further, $\mathcal{F}(f_{\lambda})$ has a unique completely invariant domain, say $U_{\lambda}$ such that $\mathcal{F}(f_{\lambda})=U_{\lambda}$ when $\lambda> \lambda_2$. Moreover, the characterization of Julia set of $f_{\lambda}$ is shown for $\lambda \in (\lambda_1, \infty)\setminus \{\lambda_2\}$.  The dynamics of $f_{\lambda}$ becomes very complicated as we decrease  $\lambda$ from  $\lambda_1$ to $0$, which is left for further research.\\ 
In Table~\ref{table1}, we provide a brief summary of the dynamics of $f_{\lambda}$ for $\lambda>0$.
\bibliographystyle{plainnat}        
\bibliography{dynamo} 
\end{document}